\documentclass[11pt,thmsa]{article}
\usepackage{amsmath, latexsym, amsfonts, amssymb, amsthm, amscd}
\usepackage{color}
\usepackage{hyperref}

\newtheorem{theorem}{Theorem}

\newtheorem{proposition}{Proposition}
\newtheorem{lemma}{Lemma}

\newtheorem{defi}{Definition}
\newcommand{\p}{\Bbb{P}}

\newcommand{\tr}{\tt {t}}

\newcommand{\R}{\mathbb{R}}

\newcommand{\ud}{\mathrm{d}}

\definecolor{dgr}{rgb}{0.72, 0.53, 0.04}

\newcommand{\Ind}[1]{\boldsymbol{1}_{\{#1\}}}
\begin{document}

\title{Backbone decomposition of multitype superprocesses.}
\author{ D. Fekete \footnote{ Department of Mathematical Sciences, University of Bath, Claverton Down, {\sc Bath, BA2 7AY, United Kingdom}. Email: d.fekete@bath.ac.uk,\, sp2236@bath.ac.uk}\, ,  S. Palau \footnotemark[1], J.C. Pardo \footnote{ {\sc Centro de Investigaci\'on en Matem\'aticas A.C. Calle Jalisco s/n. 36240 Guanajuato, M\'exico.} E-mail: jcpardo@cimat.mx, \, jl.garmendia@cimat.mx. } \, and J.L. P\'erez \footnotemark[2]}

\maketitle
\begin{abstract}
\bigskip
In this paper, we provide a construction of the so-called backbone decomposition for multitype supercritical   superprocesses.
While backbone decompositions are fairly well-known for both continuous-state branching processes and superprocesses in the one-type case, so far no such decompositions or even description of prolific genealogies have been given for the multitype cases. 

Here we focus on superprocesses, but by turning the movement off, we get the prolific backbone decomposition for multitype continuous-state branching processes as an easy consequence of our results.

\bigskip

\end{abstract}

\section{Introduction and main results.}
Motivated by the  distributional decomposition of supercritical superprocesses with quadratic branching mechanism presented in Evans and O'Connell, \cite{EO}  and the pathwise decomposition of Duquesne and Winkel \cite{DW} of continuous-state branching processes (CB-processes),  Berestycki et al. \cite{BKM} provided a pathwise construction of the so-called backbone decomposition for supercritical superprocesses.
The authors in \cite{BKM} showed that the superprocess can be written as the sum of two independent processes.
The first one is an initial burst of subcritical mass, while the second one is subcritical mass immigrating continuously and discontinuously along the path of a branching particle system called the {\it backbone} that we explain briefly below.

In Evans and O'Connell \cite{EO}  a distributional decomposition of supercritical superprocesses with quadratic spatially independent branching mechanism,   as sum of two independent processes,  was given. Later  Engl\"ander and Pinsky \cite{EP} provided a similar decomposition for the spatially dependent case.  In both constructions, the first process is a copy of the original process conditioned on extinction. The second process is understood as the aggregate accumulation of mass that has immigrated {\it continuously} along the path of an auxiliary dyadic branching particle diffusion which starts with a Poisson number of particles. Such embedded branching particle system was introduced as the {\it backbone}.

A pathwise backbone decomposition appears in Salisbury and Verzani \cite{SV}, who consider
the case of conditioning a super-Brownian motion as it exits a given domain such that the
exit measure contains at least $n$ pre-specified points in its support. There it was found
that the conditioned process has the same law as the superposition of mass that immigrates
in a Poissonian way along the spatial path of a branching particle motion which exits the
domain with precisely $n$ particles at the pre-specified points. Another pathwise backbone
decomposition for branching particle systems is given in Etheridge and Williams \cite{EW}, which
is used in combination with a limiting procedure to prove another version of Evan's immortal
particle picture.

Duquesne and Winkel \cite{DW}, in the context of L\'evy trees and with no spatial motion,  considered a similar decomposition for CB-processes whose branching mechanism $\psi$ satisfies that $0\le -\psi^\prime(0+)<\infty$ and the so-called Grey's condition
\[
\int^\infty \frac{\ud u}{\psi(u)} <\infty.
\]
 In this case the {\it backbone} corresponds to a continuous-time Galton-Watson process, and the general nature of the branching mechanism induces three different sorts of immigration. The {\it continuous immigration} is described by a Poisson point process of independent processes along the backbone, and the immigration mechanism is given by the so-called excursion measure which assigns zero initial mass and finite length to the immigration processes. The {\it discontinuous immigration} is provided by two sources of immigration. The first one is described again by a  Poisson point process of independent processes along the backbone where the immigration mechanism is given by the law of the original process conditioned on extinction, and with initial mass randomised by an infinite measure. The second  source of discontinuous immigration is given by independent copies of the original process conditioned on extinction, which are added to the backbone at its branching times, with randomly distributed initial mass that depends on the number of offspring at the branch point. 

In Berestycki et al. \cite{BKM}, a similar decomposition is provided for a class of superprocesses whose   branching mechanisms satisfy the same conditions as   those considered by Duquesne and Winkel. It is important to note that the authors in \cite{BKM} also  considered  supercritical CB-processes that, with positive probability, may die out without this ever happening in a finite time. This also allows the inclusion of branching mechanisms which are associated to CB-processes with paths of bounded variation which were excluded in \cite{DW}.  Kyprianou and Ren \cite{KR} look
at the case of a CB-process with immigration for which a similar
backbone decomposition to \cite{BKM} can be given.  Finally, backbone decompositions have also been  considered for  superprocesses with spatially dependent branching mechanisms  which are local, see  Kyprianou et al. \cite{KPeR} and Eckhoff et al. \cite{EKW}, and non-local, see Murillo-Salas and P\'erez \cite{MP} and Chen et al. \cite{CSY}. 

In this paper, we offer a similar construction for multitype  superprocesses whose  branching mechanisms are  general,  but  with the restriction of being spatially independent and having a  finite number of types.
While backbone decompositions are fairly well-known for both CB-processes and superprocesses in the one-type case, so far no such decompositions or even description of prolific genealogies (i.e. those individuals with infinite line of descent) have been given for multitype processes. Here we focus on superprocesses, but by turning the movement off, we get the prolific backbone decomposition for multitype continuous-state branching processes (MCB-processes) as an easy consequence of our results. 

Multitype superprocesses were first studied by Gorostiza and Lopez-Mimbela \cite{GL-M}  for the particular case of  quadratic branching. Later Li \cite{L92a} extended the notion of multitype superprocesses to  more general branching mechanisms (see also Section 6.2 in the monograph of Li \cite{Z}). Roughly speaking, the dynamics of the superprocesses introduced  by Li are as follows.  The movement of mass of a given  type is a Borel process,  the death and birth of mass of each type  are  associated with a spectrally positive L\'evy process. From a given type,  the creation of mass of  other types is given by  the law of a subordinator,  and  is distributed according to a discrete distribution that depends on the type. We are interested in a slightly more general superprocess where the discrete distributions are randomly chosen by a probability kernel that depends on the type. Thus the locations of non-locally displaced offspring involve two sources of randomness. One of the advantages of taking this general branching mechanism is that if there is no  spatial  motion, we recover the MCB-process  studied by Kyprianou et al. \cite{KPR}, which was properly defined by Li in Example 2.2 in \cite{Z}.

Kyprianou et al. \cite{KPR} studied the almost sure growth of supercritical MCB-processes  and implicitly described
a spine decomposition.  In \cite{KPR},  the authors show that a MCB-process  conditioned to  never get extinct is equal in law to the sum of an independent copy of the original process  and three different sources of immigration along a spine (continuous, discontinuous and in the times when the spine jumps). More precisely, the spine is given by a Markov chain,  the continuous and discontinuous immigrations  are described by a Poisson point process along the spine, where MCB-processes with the original  branching mechanism are immigrating with zero initial mass and with randomised initial mass, respectively. Due to the non-local nature of the
branching mechanism, an additional phenomenon occurs;  a positive random amount of mass immigrates off the spine each time it jumps from one state to another. Moreover, the distribution of the immigrating mass depends on where the spine jumped from and where it jumped to.

The backbone and spine decompositions  are quite different. In the backbone decomposition, the object that we dress is a multitype branching diffusion  while in the spine decomposition,  this object is a Markov chain which does not branch. Another difference is related to the immigration processes. In the spine decomposition, these are independent copies of the original process while in the backbone decomposition they are independent copies of the process conditioned to become extinct. In other words,  we can think of the backbone as all the particles that have an infinite genealogical line of descent, and of the spine as just one infinite line of descent.

\subsection{Multitype superprocesses.}

Before we introduce multitype  superprocesses and some of their properties,  we first recall some basic notation.
Let $\ell\in\mathbb{N}$ be a natural number,  and set $S=\{1,2,\cdots, \ell\}$.
We denote by $\mathcal{M}(\mathbb{R}^d)$, $\mathcal{B}(\mathbb{R}^d)$ and $\mathcal{B}^+(\mathbb{R}^d)$ the respective  spaces of finite Borel measures, bounded Borel functions and positive bounded Borel functions  on $\mathbb{R}^d$. The space $\mathcal{M}(\mathbb{R}^d)$ is endowed with the topology of  weak convergence. 

For $\boldsymbol{u},\boldsymbol{v}\in\mathbb{R}^\ell$, we introduce $[\boldsymbol{u},\boldsymbol{v}]=\sum_{j=1}^\ell u_j v_j$, and $\boldsymbol{u}\cdot \boldsymbol{v}$ as  the vector with entries $(\boldsymbol{u}\cdot \boldsymbol{v})_j = u_j v_j$. For  a  matrix $A$, we denote by $A^{\tr}$ its transpose. For any $\boldsymbol{f}=(f_1,\dots,f_\ell)^{\tr}\in\mathcal{B}(\mathbb{R}^d)^{\ell}$ and $\boldsymbol{\mu}=(\mu_1,\dots,\mu_\ell)^{\tr}\in \mathcal{M}(\mathbb{R}^d)^\ell$, we define
\[
\big\langle \boldsymbol{f},\boldsymbol{\mu}\big\rangle := \sum_{i=1}^\ell\int_{\mathbb{R}^d} f_i(x)\mu_i(\ud x).
\]
Furthermore,  we also use $| \boldsymbol{u}|:= [\boldsymbol{u}, \boldsymbol{u}]^{1/2}$ for the Euclidian norm of any $\boldsymbol{u}\in\mathbb{R}^\ell$, and $\|\boldsymbol{ \mu} \|:=\langle \boldsymbol{1},\boldsymbol{\mu}\rangle$ for the total mass of the measure $\boldsymbol{\mu}$.

Suppose that for any $i\in S$, the process $\xi^{(i)}=(\xi^{(i)}_t,t\geq 0)$ is a conservative diffusion with transition semigroup $(\mathtt{P}^{(i)}_t ,t\geq 0)$ on $\mathbb{R}^d$. We also introduce a vectorial function $\boldsymbol{\psi}:S\times \mathbb{R}^{\ell}_+\to \mathbb{R}^{\ell}$ such that 
\begin{equation}\label{branching mechanism}
\psi(i,\boldsymbol{\theta}):=- [\boldsymbol{\theta},{\boldsymbol{B}} \boldsymbol{e}_i ]+ \beta_i\theta_i^2 +\int_{\mathbb{R}^\ell_+}\left( \mathrm{e}^{- [\boldsymbol{\theta}, \boldsymbol{y}] }-1+\theta_i y_i \right)\Pi(i,\ud \boldsymbol{y}),\qquad \qquad \boldsymbol{\theta}\in\mathbb{R}^{\ell}_+, i \in S,
\end{equation}
where ${\boldsymbol{B}}$ is an  $\ell\times \ell$ real valued matrix such that ${B}_{ij}\boldsymbol{1}_{\{i\neq j\}}\in \R_+$, $\{\boldsymbol{e}_1,\dots, \boldsymbol{e}_\ell\}$  is the natural basis in $\mathbb{R}^\ell$, $\beta_i\in\mathbb{R}_+$, and $\Pi$ is a measure satisfying  the following integrability condition
\begin{equation*}
\int_{\mathbb{R}^\ell_+\setminus\{\boldsymbol{0}\}}\left((|\boldsymbol{y}|\wedge |\boldsymbol{y}|^2)+\underset{j\in S}{\sum}\Ind{j\neq i}y_j\right)\Pi(i,\ud \boldsymbol{y})<\infty, \qquad \textrm{for} \quad\ i\in S.
\end{equation*}
We call the  vectorial function $\boldsymbol{\psi}$  the branching mechanism and we also refer to  $\Pi$ as  its associated L\'evy measure.

The  first result that we present here says  that  multitype superprocesses associated to the branching mechanism $\boldsymbol{\psi}$ and the diffusions $\{\xi^{(i)}, i\in S\}$ are well-defined. Its proof is based on similar arguments as those used to prove Theorem 6.4 in Li \cite{Z}, for completeness we present its proof in Section \ref{proofs}.
 
\begin{proposition}\label{Proposition1}
	There is a strong Markov process $\boldsymbol{X}=(\boldsymbol{X}_t, (\mathcal{F}_t)_{t\ge 0},\mathbb{P}_{\boldsymbol{\mu}})$ with state space $\mathcal{M}(\mathbb{R}^d)^\ell$ and transition probabilities defined by
	\begin{align}\label{laplace}
	\mathbb{E}_{\boldsymbol{\mu}}\left[\mathrm{e}^{-\langle \boldsymbol{f},\boldsymbol{X}_t\rangle}\right]=\exp\Big\{-\langle \boldsymbol{V}_t\boldsymbol{f},\boldsymbol{\mu}\rangle\Big\},\qquad \boldsymbol{\mu}\in\mathcal{M}(\mathbb{R}^d)^\ell,
	\end{align}
	where $\boldsymbol{f}\in\mathcal{B}^+(\mathbb{R}^d)^\ell$ and  $\boldsymbol{V}_t\boldsymbol{f}(x)=({V}^{(1)}_t\boldsymbol{f}(x),\cdots,{V}^{(\ell)}_t\boldsymbol{f}(x))^{\tr}:\R^{d}\rightarrow\R_+^\ell$ 
	is the unique locally bounded solution to the vectorial integral equation
	\begin{equation}\label{inteq_u}
{V}^{(i)}_t\boldsymbol{f}(x)=\mathtt{P}^{(i)}_t f_i(x)-\int_0^t\ud s\int_{\R^d}\psi(i,\boldsymbol{V}_{t-s}\boldsymbol{f}(y))\mathtt{P}^{(i)}_s(x,\ud y),\qquad i\in S.
			\end{equation}
\end{proposition}

 \begin{defi}
The process $\boldsymbol{X}$ is called a $(\boldsymbol{\mathtt{P}},\boldsymbol{\psi})$-mutitype superprocess with $\ell$ types and with law given by $\mathbb{P}_{\boldsymbol{\mu}}$ for each initial configuration $\boldsymbol{\mu}\in \mathcal{M}(\mathbb{R}^d)^\ell$. 
\end{defi}

Our definition is consistent with the multitype superprocesses that appear in the literature. Indeed, we observe that the multitype superprocesses considered  by  Gorostiza and Lopez-Mimbela \cite{GL-M} are associated 
with the branching mechanism
$$	\psi(i,\boldsymbol{\theta})=- d_i[\boldsymbol{\theta},
\boldsymbol{\pi}^{(i)}]+ \beta_i\theta_i^2, $$
where $d_i,\beta_i\in\R_+$,  $\boldsymbol{\pi}^{(i)}=\{\pi^{(i)}_j, j\in S\}$ is a probability distribution on $S$, and the spatial movement is driven by the family $\{\xi^{(i)}, i\in S\}$ of symmetric stable processes.
 Li \cite{L92a} (see also Section 6.2 in \cite{Z}) introduced  multitype superprocesses with spatial movement driven by  Borel right processes and whose branching mechanism is of the form
\begin{equation*}
\begin{split}
\psi(i,\boldsymbol{\theta})&=b_{i}\theta_i+\beta_i\theta_i^2-d_{i}[\boldsymbol{\theta}, \boldsymbol{\pi}^{(i)}]+\int_{\mathbb{R}_+}\left( \mathrm{e}^{- u\theta_i }-1+\theta_i u \right)l(i,\ud u)+\int_{\mathbb{R}_+}\left( \mathrm{e}^{- u[\boldsymbol{\theta}, \boldsymbol{\pi}^{(i)}] }-1 \right)n(i,\ud u),
\end{split}
\end{equation*}
where $b_i,d_i,\beta_i\in\R_+$, $\boldsymbol{\pi}^{(i)}=\{\pi^{(i)}_j, j\in S\}$ is a probability distribution on $S$, and $l(i,\ud u)$, $ n(i,\ud u)$ are  measures on $\R_+$ satisfying  
\[
\int_{\R_+}(u\wedge u^2)l(i,\ud u)<\infty\qquad  \textrm{and}\qquad  \int_{\R_+}u \,n(i,\ud u)<\infty,
\]
that represent the local and non-local kernels, respectively. The latter branching mechanism can be rewritten in the form of  \eqref{branching mechanism} by taking $B_{ji}:=-b_i\Ind{i=j}+d_i\pi^{(i)}_j$,  and 
$$\Pi(i,\ud \boldsymbol{y})=\Ind{\boldsymbol{y}=u\boldsymbol{e}_i}l(i,\ud u)+\Ind{\boldsymbol{y}=u\boldsymbol{\pi}^{(i)}}n(i,\ud u).$$

It is important to note that if the branching mechanism is given as in \eqref{branching mechanism} and there is no  spatial movement, then the associated total mass of a superprocess is a  MCB-process,  see for instance Example 2.2 in \cite{Z}. Indeed, it is not difficult to see that the total mass vector of a multitype superprocess is a MCB-process. Recall that an $\ell$-type MCB-process $\boldsymbol{Y}=(\boldsymbol{Y}_t,t\geq 0)$  with branching mechanism $\boldsymbol{\psi}$ can be characterised through its Laplace transform.
If we denote by $\mathbf{P}_{\boldsymbol{y}}$ the law of such a process with initial state $\boldsymbol{y}\in\mathbb{R}_+^\ell$, then 
\begin{equation}\label{ec mcsb}
\mathbf{E}_{\boldsymbol{y}}\left[ \mathrm{e}^{-[\boldsymbol{\theta},\boldsymbol{Y}_t]}\right]=\exp\big\{- [\boldsymbol{y},\boldsymbol{v}_t(\boldsymbol{\theta})]\big\},\qquad \textrm{for }\quad\boldsymbol{\theta}\in\mathbb{R}_+^\ell, t\ge 0,
\end{equation}
where 
\[
t\mapsto\boldsymbol{v}_t(\boldsymbol{\theta})=(\boldsymbol{v}_t(1,\boldsymbol{\theta})\dots,\boldsymbol{v}_t(\ell, \boldsymbol{\theta}))^{\tr}
\]
is the unique locally bounded  solution, with non-negative entries,  to the system of integral equations
\begin{equation}\label{ec v for mcsb}
\boldsymbol{v}_t(i,\boldsymbol{\theta})=\theta_i-\int_0^t \psi(i, \boldsymbol{v}_{t-s}(\boldsymbol{\theta}))\mathrm{d}s, \quad i\in S.
\end{equation}
Suppose that $(\boldsymbol{X}_t,\p_{\boldsymbol{\mu}})_{t\geq 0}$ is a $(\boldsymbol{\mathtt{P}},\boldsymbol{\psi})$-multitype superprocess and define the total mass vector as  $\boldsymbol{Y}=(\boldsymbol{Y}_t,t\geq 0)$  with entries 
$$Y_t(i)=X_t(i,\R^d)=\int_{\R^d}	X_t(i,\ud x), \qquad t\geq 0,$$
 and initial vector $\boldsymbol{\mu}=(\mu_1(\R^d),\cdots,\mu_{\ell}(\R^d))^{\tr}$. 
Let $\boldsymbol{\theta}\in \mathbb{R}_+^{\ell}$,  and take  $f_i(x)=\theta_i$ for each $i\in S, x\in \R^d$.  Since the branching mechanism  and the vector $\boldsymbol{\theta}$ are spatially independent, the system of functions $\boldsymbol{V}_t\boldsymbol{\theta}$ that satisfies \eqref{inteq_u} does not depend on $x\in\mathbb{R}^d$. In other words
\[
\begin{split}
{V}_t^{(i)}\boldsymbol{\theta}&=\mathtt{P}^{(i)}_t\theta_i-\int_0^t\ud s\int_{\R^d}\psi(i,\boldsymbol{V}_{t-s}\boldsymbol{\theta})\mathtt{P}^{(i)}_s(x,\ud y)\\
&=\theta_i-\int_0^t\psi(i,\boldsymbol{V}_{t-s}\boldsymbol{\theta})\ud s, \quad i\in S.
\end{split}
\]
Recall that the previous system of equations  has a unique solution, therefore $\boldsymbol{V}_t\boldsymbol{\theta}=\boldsymbol{v}_t(\boldsymbol{\theta})$ for any $x\in\mathbb{R}^d$. By \eqref{laplace} and the relationship between  $\boldsymbol{X}$ and $\boldsymbol{Y}$, the total mass vector is indeed a MCB-process.

Since the   total mass vector of a multitype superprocess is a MCB-process, we can determine its asymptotic behaviour  through its first moment, similarly to  the one-type case. More precisely, 
denote by $\boldsymbol{M}(t)$ the $\ell\times \ell$ matrix with elements
\[
{M}(t)_{ij}=\mathbb{E}_{\boldsymbol{e}_i \delta_x}\Big[\langle \boldsymbol{e}_j, \boldsymbol{X}_t\rangle\Big], \quad i,j\in S,
\]
where $\boldsymbol{e}_i\delta_x$ denotes a measure valued vector that has unit mass at position $x\in\mathbb{R}^d$, in the $i$-th coordinate, and  zero mass everywhere else.

Barczy et al. \cite{BLP} (see Lemma 3.4) proved that the mean matrix $\boldsymbol{M}(t)$ can be written in terms of the branching mechanism $\boldsymbol{\psi}$. In other words, for all $t>0$
$$ \boldsymbol{M}(t)=\mathrm{e}^{-t\widetilde{\boldsymbol{B}}^{\tr} },$$ where the matrix $\widetilde{\boldsymbol{B}}$ is given by
$$\widetilde{B}_{i,j}={B}_{i,j}+\int_{\mathbb{R}^\ell_+}(y_i-\delta_{i,j})^+\Pi(j,\ud \boldsymbol{y}),$$
where $(a)^+=a\lor 0$, denotes the positive part of $a$.  Moreover, after straightforward computations (see for instance the computations after identity (2.15) in \cite{BLP}) we observe that the branching mechanism $\boldsymbol{\psi}$ can be rewritten as 
follows
\begin{equation}\label{psiBtilde}
\psi(i,\boldsymbol{\theta}):=- [\boldsymbol{\theta},\widetilde{\boldsymbol{B}} \boldsymbol{e}_i ]+ \beta_i\theta_i^2 +\int_{\mathbb{R}^\ell_+}\left( \mathrm{e}^{- [\boldsymbol{\theta}, \boldsymbol{y}] }-1+[\boldsymbol{\theta}, \boldsymbol{y}] \right)\Pi(i,\ud \boldsymbol{y}),\qquad \qquad \boldsymbol{\theta}\in\mathbb{R}^{\ell}_+, i \in S.
\end{equation}
In the sequel, we assume that the matrix $\widetilde{\boldsymbol{B}}^{\tr}$ is irreducible,  and therefore the mean matrix $\boldsymbol{M}$ is irreducible as well. Then classical Perron-Frobenius theory guarantees that there exists  a unique leading eigenvalue $\Gamma$, and  right  and left eigenvectors ${\tt u},  {\tt v}\in\R_+^\ell$,   whose coordinates are  strictly positive such that, for $t\ge 0,$
\[
\boldsymbol{M}(t){\tt u}=\mathrm{e}^{\Gamma t}{\tt u}, \qquad \widetilde{\boldsymbol{B}}^{\tr}{\tt u}=\Gamma{\tt u}, \qquad {\tt v}^{\tr}\boldsymbol{M}(t)=\mathrm{e}^{\Gamma t}{\tt v}^{\tr},  \qquad \textrm{and}\qquad  {\tt v}^{\tr}\widetilde{\boldsymbol{B}}^{\tr}=\Gamma {\tt v}^{\tr}.
\]
It is important to note,  that since the branching mechanism is spatially independent,  the value of $\Gamma$ does not depend on the spatial variable. 

Moreover, $\Gamma$  determines the long term behaviour of $\boldsymbol{X}$. Indeed, employing the same terminology as in the one-type  case,  we call the process supercritical, critical or subcritical accordingly as $\Gamma$ is strictly positive, equal to zero or strictly negative.  In Kyprianou and Palau \cite{KP},  the authors show  that when $\Gamma\leq 0$ the total mass goes to zero almost surely.  Barczy and Pap \cite{BP} show that if $\Gamma>0$, then the total mass process satisfies 
	$$\underset{t\rightarrow\infty}{\lim} \mathrm{e}^{-\Gamma t}\ \mathbf{E}_{\boldsymbol{e}_i}\left[ \boldsymbol{Y}_t\right]=\boldsymbol{e}_i^{\tr} {\tt u}{\tt v}^{\tr},\qquad \mbox{for } i\in S,$$
 which is a non-zero vector.  In particular, we deduce that 
 \begin{equation}\label{absortionle1}
 \mathbb{P}_{\boldsymbol{e}_i\delta_{x}}\Big(\lim_{t\rightarrow\infty}\|\boldsymbol{X}_t\|=0\Big)<1, \qquad \mbox{for } i\in S,\quad x\in \mathbb{R}^d.
 \end{equation} 

Here, we are also interested in the case that  extinction occurs in finite time. More precisely,  let us define $\mathcal{E}:=\{\|\boldsymbol{X}_t\|=0 \mbox{ for some } t>0 \}$, the event of {\it extinction}
and take $w_i:\mathbb{R}^d\mapsto \mathbb{R}_+$ be the function such that
\begin{equation}\label{w} 
w_i(x):=-\log \mathbb{P}_{\boldsymbol{e}_i\delta_x}(\mathcal{E}),\quad i\in S.
\end{equation}
Since the branching mechanism is spatially independent, and  the total mass vector of $\boldsymbol{X}$ is a MCB-process,  we get that $w_i(x)=w_i$, for all $x\in\mathbb{R}^d$, for some constant $w_i$.

 In what follows, we assume 
 \begin{equation}
	\label{extinction assumption}
	0<w_i<\infty, \qquad \mbox{ for all } i\in S.
	\end{equation}
 Assumption \eqref{extinction assumption} or similar assumptions have been used in most of the cases where backbones have been constructed. For instance in \cite{BKM} and \cite{DW}, the authors assume Grey's condition which is equivalent to $w_i$ being finite. In \cite{CSY,EKW, KPR, MP} a very similar  condition appears for the spatially dependent case. Assumption \eqref{extinction assumption} is not only necessary for the construction of the multitype superprocess conditioned on extinction but also for the construction of the so-called Dynkin-Kuznetsov measure, as we will see below.
 
 On the other hand,  assumption \eqref{extinction assumption} is not very restrictive. For instance, it is satisfied if  $\Gamma> 0$ and $\beta:=\inf_{i\in S}\beta_i>0$. Indeed  from \eqref{absortionle1}, we see that 
	$$\mathbb{P}_{\boldsymbol{e}_i\delta_x}(\mathcal{E})\leq\mathbb{P}_{\boldsymbol{e}_i\delta_x}\Big(\lim_{t\rightarrow\infty}\|\boldsymbol{X}_t\|=0\Big)<1.$$
	From \eqref{ec mcsb} and the fact that the total mass is a MCB-process, it is clear that 
	$$\mathbb{P}_{\boldsymbol{e}_i\delta_x}\Big(\|\boldsymbol{X}_t\|=0\Big)=\exp\left\{-\lim_{\boldsymbol{\theta}\hookrightarrow\infty} v_t(i,\boldsymbol{\theta})\right\},$$
	where $\boldsymbol{v}_t(i,\boldsymbol{\theta})$ is  given by \eqref{ec v for mcsb} and $\boldsymbol{\theta}\hookrightarrow\infty$ means that each coordinate of $\boldsymbol{\theta}$ goes to $\infty$. In other words, if we show that 
	$$\lim_{t\rightarrow\infty}\lim_{\boldsymbol{\theta}\hookrightarrow\infty} v_t(i,\boldsymbol{\theta})<\infty \qquad \mbox{ for all } i\in S,$$
	then we have that  \eqref{extinction assumption} holds. In order to prove that the above limit is finite, we introduce  
	$$A_t(\boldsymbol{\theta}):=\sup_{i\in S}\frac{v_t(i,\boldsymbol{\theta})}{{\tt u}_i},$$
	where ${\tt u}_{i}$ denotes the $i$-th coordinate of the right eigenvector associated to $\Gamma$. Since the supremum of finitely many continuously differentiable functions is differentiable except at  most countably many isolated points, we  may fix $t\geq 0$ such that $A_t(\boldsymbol{\theta})$ is differentiable  at $t$ and select $i$ in such a way that $A_t(\boldsymbol{\theta}){\tt  u}_{i}=v_t(i,\boldsymbol{\theta})$.
	Then by using \eqref{ec v for mcsb} and  \eqref{psiBtilde}  we can deduce that
	\begin{equation*}
	\begin{split}
	{\tt u}_{i}\frac{\ud}{\ud t}A_t(\boldsymbol{\theta})=&\frac{\ud}{\ud t}v_t(i,\boldsymbol{\theta})=\underset{j\in S}{\sum}\widetilde
	{B}_{ji}{\tt u}_j\frac{\boldsymbol{v}_t(j,\boldsymbol{\theta})}{{\tt u}_j}- \beta_i(\boldsymbol{v}_t(i,\boldsymbol{\theta}))^2 -\int_{\mathbb{R}^\ell_+}\left( \mathrm{e}^{- [\boldsymbol{v}_t(\boldsymbol{\theta}), \boldsymbol{y}] }-1+[\boldsymbol{v}_t(\boldsymbol{\theta}), \boldsymbol{y}] \right)\Pi(i,\ud \boldsymbol{y}).
	\end{split}
	\end{equation*}
	Since $1-x-\mathrm{e}^{-x}\leq 0$, for all $x>0$, $\widetilde{B}_{i,j}\Ind{i\neq j}>0$ and $A_t(\boldsymbol{\theta}){\tt u}_i=v_t(i,\boldsymbol{\theta})$, we have 
	$${\tt u}_{i}\frac{\ud}{\ud t}A_t(\boldsymbol{\theta})\leq A_t(\boldsymbol{\theta})\underset{j\in S}{\sum}\widetilde
	{B}_{ji}{\tt u}_{j}- \beta_i(A_t(\boldsymbol{\theta}){\tt u}_{i})^2 =A_t(\boldsymbol{\theta})(\widetilde{\boldsymbol{B}}^{\tr}{\tt u})_i- \beta_i(A_t(\boldsymbol{\theta}){\tt u}_{i})^2.$$
	Next, we use that ${\tt u}$ is an eigenvector of $\widetilde{\boldsymbol{B}}^{\tr}$ to get
	$${\tt u}_{i}\frac{\ud}{\ud t}A_t(\boldsymbol{\theta})\leq A_t(\boldsymbol{\theta})\Gamma {\tt u}_{i}- \beta_i(A_t(\boldsymbol{\theta}){\tt u}_{i})^2.$$
	By defining $\underline{{\tt u}}:=\inf_{i\in S}{\tt u}_{i}$ and recalling the definition of $\beta$, the previous identity implies
	$$\frac{\ud}{\ud t}A_t(\boldsymbol{\theta})\leq A_t(\boldsymbol{\theta})\Gamma- \beta\underline{{\tt u}}(A_t(\boldsymbol{\theta}))^2.$$
	Since, $\Gamma, \beta$ and $\underline{{\tt u}}$ are strictly  positive, an integration by parts allow us to deduce that 
	$$A_t(\boldsymbol{\theta})\leq \frac{\Gamma\mathrm{e}^{\Gamma t}}{\frac{\Gamma}{A_0(\boldsymbol{\theta})}+\beta\underline{{\tt u}}(\mathrm{e}^{\Gamma t}-1)}.$$
	Finally, if we define  $\overline{\tt u}:=\sup_{i\in S}{\tt u}_i$,  the  previous computations lead  to 
	$$w_i=\lim_{t\rightarrow\infty}\lim_{\boldsymbol{\theta}\hookrightarrow\infty} v_t(i,\boldsymbol{\theta})\leq \overline{\tt u}\lim_{t\rightarrow\infty}\lim_{\boldsymbol{\theta}\hookrightarrow\infty}A_t(\boldsymbol{\theta})\leq \frac{\overline{\tt u}\Gamma}{\beta\underline{\tt u}}<\infty.$$

 The following result is also needed for  constructing  the associated Dynkin-Kuznetsov measures  which provide a way to dress the backbone.
	\begin{proposition}\label{finite extinction}
		Suppose that condition \eqref{extinction assumption} holds, then $\boldsymbol{\psi}(\boldsymbol{w})=\boldsymbol{0}$. Moreover, for $x\in\R^d,\ i\in S $ and $t>0$, we have that 
		\begin{equation*}
		\mathbb{P}_{\boldsymbol{e}_i\delta_x}(\|\boldsymbol{X}_t\|=0 )>0.
		\end{equation*}
	\end{proposition}
For simplicity of exposition, the proof of this result is presented in Section \ref{proofs}.

As we said before, our aim is to describe the backbone decomposition of   $\boldsymbol{X}$. According to Berestycki et al. \cite{BKM} a one-type  supercritical superprocess  can be decomposed into an initial burst of subcritical mass and three types of immigration processes  along the backbone, which are two types of Poissonian immigrations  and branch point immigrations. 
In order to use the same idea in the multitype case, we need to determine the components of this decomposition. These  are the multitype branching diffusion process,  that gives the prolific genealogies, and copies of the original multitype superprocess conditioned on extinction.

\subsection{The multitype supercritical superdiffusion conditioned on extinction.}

It is well known that under some conditions a supercritical CB-process can be conditioned to become extinct  by conditioning the associated spectrally positive L\'evy process to drift to $-\infty$. 
Such a conditioning appears as an Esscher transform on the underlying L\'evy process in the Lamperti transform, where the shift parameter is given by the largest root of the branching mechanism. 
Here we show that a similar result still holds in the multitype case. 
In particular we have the following result.

\begin{proposition}\label{proposition2}
	For each $\boldsymbol{\mu}\in\mathcal{M}(\mathbb{R}^d)^\ell$, define the law of $\boldsymbol{X}$ with initial configuration $\boldsymbol{\mu}$ conditioned on becoming extinct by $\mathbb{P}^\dag_{\boldsymbol{\mu}}$, and let $\mathcal{F}_t:=\sigma(X_s,s\leq t)$. Specifically, for all events $A$, measurable with respect to $\mathcal{F}$,
	\[
	\mathbb{P}^\dag_{\boldsymbol{\mu}}(A)=\mathbb{P}_{\boldsymbol{\mu}}\left(A\left|\mathcal{E}\right.\right).
	\]
	Then, for all $\boldsymbol{f}\in\mathcal{B}^+(\mathbb{R}^d)^\ell$ 
	\[
	\mathbb{E}^\dag_{\boldsymbol{\mu}}\left[\mathrm{e}^{-\langle \boldsymbol{f},\boldsymbol{X}_t\rangle}\right]=\exp\Big\{-\langle \boldsymbol{V}^\dag_t\boldsymbol{f},\boldsymbol{\mu}\rangle\Big\},
	\]
	where
	\[
{V}^{\dag,(i)}_t \boldsymbol{f}(x):={V}^{(i)}_t (\boldsymbol{f}+\boldsymbol{w})(x)-w_i, \qquad i\in S,
	\]
	is the unique locally bounded solution of
	\begin{align}\label{semi-cond}
	{V}^{\dag,(i)}_t \boldsymbol{f}(x)=\mathtt{P}^{(i)}_t f_i(x)-\int_0^t\ud s\int_{\R^d}\psi^\dag(i,\boldsymbol{V}^\dag_{t-s}\boldsymbol{f}(y))\mathtt{P}^{(i)}_s(x,\ud y), \qquad i\in S,
	\end{align}
	where $\boldsymbol{\psi}^\dag(\boldsymbol{\lambda}):=\boldsymbol{\psi}(\boldsymbol{\lambda}+\boldsymbol{w})$ and $\boldsymbol{w}$ is given by \eqref{w}. 
In other words, $(\boldsymbol{X},\mathbb{P}_{\boldsymbol{\mu}}^\dag)$ is a $(\boldsymbol{\mathtt{P}},\boldsymbol{\psi}^\dag)$-multitype superprocess.
\end{proposition}
For simplicity of exposition, the proof of this result is presented in Section \ref{proofs}.

\subsection{Dynkin-Kuznetsov measure.}
As we mentioned before, a key ingredient in the construction of the backbone, or even the spine  decomposition for superprocesses, is the so-called  Dynkin-Kuznetsov measure. It is important to note that the existence of such  measures was taken for granted in most of  the references that appear in the literature, in particular in \cite{BKM, EKW, KPR, MP}. Fortunately, from the assumptions and the way the dressing processes are constructed this omission does not play an important role on the validity of their results. 
 Here, we  provide  a rigorous argument  for their existence. 
 See also Chen et al.  \cite{CSY} for the study of Dynkin-Kutznetsov measures for one-type superprocesses with non-local branching mechanism.
  
 Let us denote by  $\mathcal{X}$   the space  of c\`adl\`ag paths  from $[0,\infty)$ to $\mathcal{M}(\mathbb{R}^d)^\ell$. 

 \begin{proposition}\label{Dyn-Kuz} Let $\boldsymbol{X}$ be a $(\boldsymbol{\mathtt{P}},\boldsymbol{\psi})$-multitype superprocess satisfying \eqref{extinction assumption}. For $x\in \mathbb{R}^d$, there exists a measure $\mathbb{N}_{x\boldsymbol{e}_i}$  on the space $\mathcal{X}$  satisfying
\begin{equation}\label{NmeasureLaplace}
\mathbb{N}_{x\boldsymbol{e}_i}\left(1-\mathrm{e}^{-\langle \boldsymbol{f},\boldsymbol{X}_t\rangle}\right)=-\log\mathbb{E}_{\delta_x\boldsymbol{e}_{i}}\left[\mathrm{e}^{-\langle \boldsymbol{f},\boldsymbol{X}_t\rangle}\right],
\end{equation}
for all $\boldsymbol{f}\in \mathcal{B}(\R^d)^\ell$ and $t\geq0$.
 \end{proposition}
 Again, for simplicity of exposition  we provide  the proof of this Proposition  in Section \ref{proofs}.

Following the same terminology as  in the literature, we call  $\{(\mathbb{N}_{x\boldsymbol{e}_i}, x\in \R^d), i\in S\}$ the Dynkin-Kuznetsov measures. We denote by $\mathbb{N}^\dag$  the Dynkin-Kuznetsov measures associated to the multitype superprocess conditioned on extinction,  which are also well defined (see the discussion after the proof of Proposition \ref{Dyn-Kuz}).

\subsection{Prolific individuals.}

Here,  we consider those individuals of the superprocess who are responsible for the infinite growth of the process. In our  case,  we have that the so-called  prolific individuals, i.e. those with an infinite genealogical line of descent,  form a branching particle diffusion where the particles move according to the same motion semigroup as the superprocess itself,  and their branching generator can be expressed in terms of the branching mechanism of the superprocess.
Let $\boldsymbol{Z}=(\boldsymbol{Z}_t, t\geq 0)$  be a multitype branching diffusion process (MBDP) with $\ell$ types, where the movement of each particle of type $i\in S$ is given by the semigroup $\mathtt{P}^{(i)}$.
The branching rate $\boldsymbol{q}\in\mathbb{R}^\ell_+$ takes the form  
\begin{equation}\label{brate}
	q_i=\frac{\partial}{\partial x_i}\psi(i, \boldsymbol{w}), \qquad\text{$i\in S$},
\end{equation}
where $\boldsymbol{w}$ was defined in  \eqref{w}.

The offspring distribution $(p^{(i)}_{j_1,\dots, j_\ell})_{(j_1,\dots, j_\ell)\in\mathbb{N}^\ell}$ satisfies
\begin{equation}\label{p^i definition}
	\begin{split}
		p^{(i)}_{j_1,\dots,j_\ell}&=\frac{1}{w_iq_i}\left( \beta_i w_i^2\boldsymbol{1}_{\{\boldsymbol{j}=2\boldsymbol{e}_i\}}+\left( B_{ki} w_k+\int_{\mathbb{R}_+^\ell}w_k y_k \mathrm{e}^{-[\boldsymbol{w},\boldsymbol{y}]}\Pi(i,\mathrm{d}\boldsymbol{y})\right)\boldsymbol{1}_{\{\boldsymbol{j}=\boldsymbol{e}_k\}}\boldsymbol{1}_{\{i\neq k\}}\right. \\
		&\hspace{4cm}\left.+\int_{\mathbb{R}_+^\ell}\frac{(w_1 y_1)^{j_1}\dots(w_\ell y_\ell)^{j_\ell}}{j_1!\dots j_\ell !}\mathrm{e}^{-[ \boldsymbol{w},\boldsymbol{y}]}\Pi(i,\mathrm{d}\boldsymbol{y}) \boldsymbol{1}_{\{j_1+\dots+j_\ell\geq 2\}}\right),
	\end{split}
\end{equation}
where $\boldsymbol{j}=(j_1,\cdots, j_\ell)$. Note that $p^{(i)}_{j_1,\dots,j_\ell}$ is a probability distribution.  Indeed,  since  $\boldsymbol{ \psi}(\boldsymbol{w})=\boldsymbol{ 0}$, for each $i\in S$ we get that  
\begin{equation*}
	\begin{split}
		w_iq_i&=w_iq_i-\psi(i,\boldsymbol{w})\\
		&=w_i\left(-B_{ii}+2\beta_i w_i+\int_{\mathbb{R}_+^\ell}\left( 1-\mathrm{e}^{-[ \boldsymbol{w},\boldsymbol{y}]}\right)y_i\Pi(i,\mathrm{d}\boldsymbol{y})\right)\\
		&\hspace{6cm}+[ \boldsymbol{w},\boldsymbol{B}\boldsymbol{e}_i] -\beta_i w_i^2-\int_{\mathbb{R}_+^\ell}\left( \mathrm{e}^{-[ \boldsymbol{w},\boldsymbol{y}]}-1+w_i y_i\right)\Pi(i,\mathrm{d}\boldsymbol{y})\\
		&=\sum_{j\neq i}B_{ji} w_j+\beta_i w_i^2+\int_{\mathbb{R}_+^\ell}\mathrm{e}^{-[ \boldsymbol{w},\boldsymbol{y}]}\left( \mathrm{e}^{[ \boldsymbol{w},\boldsymbol{y} ]}-1-w_iy_i \right)\Pi(i,\mathrm{d}\boldsymbol{y})\\
		&=\sum_{j\neq i}\left(B_{ji} w_j+\int_{\mathbb{R}_+^\ell}w_j y_j \mathrm{e}^{-[ \boldsymbol{w},\boldsymbol{y}]}\Pi(i,\mathrm{d}\boldsymbol{y})\right)+\beta_i w_i^2+\int_{\mathbb{R}_+^\ell}\mathrm{e}^{-[ \boldsymbol{w},\boldsymbol{y}]}\left( \mathrm{e}^{[ \boldsymbol{w},\boldsymbol{y}]}-1-[ \boldsymbol{w},\boldsymbol{y}]\right)\Pi(i,\mathrm{d}\boldsymbol{y})\\
		&=\sum_{j\neq i}\left(B_{ji} w_j+\int_{\mathbb{R}_+^\ell}w_j y_j \mathrm{e}^{-[ \boldsymbol{w},\boldsymbol{y}]}\Pi(i,\mathrm{d}\boldsymbol{y})\right)+\beta_i w_i^2\\
		&\hspace{7cm}+\int_{\mathbb{R}_+^\ell}\sum_{j_1+\dots+j_\ell\geq 2}\frac{(w_1 y_1)^{j_1}\dots(w_\ell y_\ell)^{j_\ell}}{j_1! \dots j_\ell !}\mathrm{e}^{-[ \boldsymbol{w},\boldsymbol{y}]}\Pi(i,\mathrm{d}\boldsymbol{y}),
	\end{split}
\end{equation*}
where in the last row we have used the multinomial theorem, i.e.
\begin{equation}\label{eq:binom}
\begin{split}
\sum_{n=2}^{\infty}\frac{[\boldsymbol{x},\boldsymbol{y}]^n}{n!}&=\sum_{n=2}^\infty \frac{1}{n!} \sum_{j_1+\dots+j_\ell=n}\binom{n}{j_1,\dots,j_\ell}\prod_{k=1}^\ell(x_k y_k)^{j_k}=\sum_{j_1+\dots+j_\ell\geq 2}\frac{(x_1 y_1)^{j_1}\dots(x_\ell y_\ell)^{j_\ell}}{j_1! \dots j_\ell!}.
\end{split}
\end{equation}
Let $\boldsymbol{F}(\boldsymbol{s})=(F_1(\boldsymbol{s}),\dots,F_\ell(\boldsymbol{s}))^{\tr}$, $\boldsymbol{s}\in [0,1]^\ell$, be the branching mechanism of   $\boldsymbol{Z}$,  which is given by
\begin{equation}\label{Fi}
F_i(\boldsymbol{s})=q_i\sum_{\boldsymbol{j}\in\mathbb{N}^\ell}(s_1^{j_1}\dots s_\ell^{j_\ell}-s_i)p_{j_1,\dots,j_\ell}^{(i)}=\frac{1}{w_i}\psi(i,\boldsymbol{w}\cdot(\boldsymbol{1}-\boldsymbol{s})),\qquad i\in S,
\end{equation}
where we recall that  $\boldsymbol{1}$ denotes the vector with value 1 in each coordinate and $\boldsymbol{u}\cdot \boldsymbol{v}$ is the element-wise  multiplication of  the vectors $\boldsymbol{u}$ and $\boldsymbol{v}$.
The intuition behind the process $\boldsymbol{Z}$ is as follows. 
A particle of type $i$ from  its birth executes a $\mathtt{P}^{(i)}$ motion, and  after an independent and exponentially distributed random time with parameter $q_i$ dies and gives birth at its death position  to an independent number of offspring with distribution $\{ p_{\boldsymbol{j}}^{(i)},\boldsymbol{j}\in\mathbb{N}^\ell \}$. 
We call $\boldsymbol{Z}$ the {\it backbone} of the multitype superprocess $\boldsymbol{X}$,  and denote its initial distribution by  $\boldsymbol{\nu}\in\mathcal{M}_a(\mathbb{R}^d)^\ell$, where  $\mathcal{M}_a(\mathbb{R}^d)$ denotes the space of  atomic measures on $\mathbb{R}^d$.
Comparing the form of the offspring distribution between the one-type case and the  multitype case, the main difference is that now we are allowed to have one offspring at a branching event. However  in this case,  that offspring  has to have a  different type from its parent.

\subsection{The backbone decomposition.}
Our primary aim is to  give a decomposition of the $(\boldsymbol{\mathtt{P}},\boldsymbol{\psi})$-multitype superprocess along its embedded backbone $\boldsymbol{Z}$.
The main idea is to dress  the process $\boldsymbol{Z}$ with immigration, where the processes  we immigrate are copies of the $(\boldsymbol{\mathtt{P}},\boldsymbol{\psi}^\dag)$-multitype superprocess. 
The dressing relies on three different types of immigration mechanisms.
These  are two types of Poissonian immigrations along the life span of each prolific individual, and  an additional creation of mass at the branch points of the  embedded particle system. In the first case, we immigrate independent copies of the $(\boldsymbol{\mathtt{P}},\boldsymbol{\psi}^\dag)$-multitype superprocess,  where the immigration rate along a particle of  type $i\in S$ is related to a subordinator in $\mathbb{R}^\ell_+$, whose Laplace exponent is given by
\[
\phi(i,\boldsymbol{\lambda})=\frac{\partial}{\partial x_i}\psi^\dag (i,\boldsymbol{\lambda})-\frac{\partial}{\partial x_i}\psi^\dag (i,\boldsymbol{0})=\frac{\partial}{\partial x_i}\psi(i,\boldsymbol{\lambda}+\boldsymbol{w})-\frac{\partial}{\partial x_i}\psi(i,\boldsymbol{w}),
\]
which can be rewritten as
\begin{equation}\label{phi}
\phi(i,\boldsymbol{\lambda})=2\beta_i\lambda_i+\int_{\mathbb{R}_+^\ell}\left(1-\mathrm{e}^{-[\boldsymbol{\lambda},\boldsymbol{y}]} \right)y_i\mathrm{e}^{-[\boldsymbol{w},\boldsymbol{y}]}\Pi(i,\mathrm{d}\boldsymbol{y}).
\end{equation}
When an individual of type $i\in S$ has branched and its offspring is given by $\boldsymbol{j}=(j_1,\dots,j_\ell)\in\mathbb{N}^\ell$,  we immigrate an independent copy of the  $(\boldsymbol{\mathtt{P}},\boldsymbol{\psi}^\dag)$-multitype superprocess  where the initial mass has distribution 
\begin{equation}\label{eq:eta}
\begin{split}
\eta^{(i)}_{\boldsymbol{j}}(\mathrm{d}\boldsymbol{y})=\frac{1}{w_iq_i p^{(i)}_{\boldsymbol{j}}}&\bigg( \beta_i w_i^2\boldsymbol{1}_{\{\boldsymbol{j}=2\boldsymbol{e}_i\}}\delta_{\boldsymbol{0}}(\mathrm{d}\boldsymbol{y})+\left( B_{ki} w_k\delta_{\boldsymbol{0}}(\mathrm{d}\boldsymbol{y})+w_k y_k \mathrm{e}^{-[\boldsymbol{w},\boldsymbol{y}]}\Pi(i,\mathrm{d}\boldsymbol{y})\right)\boldsymbol{1}_{\{\boldsymbol{j}=\boldsymbol{e}_k\}}\boldsymbol{1}_{\{i\neq k\}} \\
		&\hspace{4cm}+\frac{(w_1 y_1)^{j_1}\dots(w_\ell y_\ell)^{j_\ell}}{j_1!\dots j_\ell !}\mathrm{e}^{-[ \boldsymbol{w},\boldsymbol{y}]}\Pi(i,\mathrm{d}\boldsymbol{y}) \boldsymbol{1}_{\{j_1+\dots+j_\ell\geq 2\}}\bigg).
\end{split}
\end{equation}

Before we state our main results, we recall and introduce some notation.
Recall that $\mathcal{X}$ denotes the space of c\`adl\`ag paths.  Similarly to  the one-type case, we  use an  Ulam-Harris labelling to reference the particles, and we denote the obtained tree by $\mathcal{T}$.
For a particle $u\in\mathcal{T}$ let $\gamma_u$ denote the type of the particle, $\tau_u$ its  birth time, $\sigma_u$ its  death time,  and $z_u(t)$ its  spatial position at time $t$ (whenever  $\tau_u\leq t<\sigma_u$).

\begin{defi}
For $\boldsymbol{\nu}\in \mathcal{M}_a(\mathbb{R}^d)^\ell$, let $\boldsymbol{Z}$ be a MBDP with initial configuration $\boldsymbol{\nu}$,  and let  $\widetilde{\boldsymbol{X}}$ be an independent copy of $\boldsymbol{X}$ under $\mathbb{P}_{\boldsymbol{\mu}}^\dag$.
We define the stochastic process $\boldsymbol{\Lambda}=(\boldsymbol{\Lambda}_t,t\geq0)$  on $\mathcal{M}(\mathbb{R}^d)^\ell$ by
\[
\boldsymbol{\Lambda}=\widetilde{\boldsymbol{X}}+\boldsymbol{I}^{\mathbb{N}^\dag}+\boldsymbol{I}^{\mathbb{P}^\dag}+\boldsymbol{I}^{\eta},
\]
where the processes $\boldsymbol{I}^{\mathbb{N}^\dag}=(\boldsymbol{I}^{\mathbb{N}^\dag}_t,t\geq0)$, $\boldsymbol{I}^{\mathbb{P}^\dag}=(\boldsymbol{I}^{\mathbb{P}^\dag}_t,t\geq0)$, and $\boldsymbol{I}^{\eta}=(\boldsymbol{I}^{\eta}_t,t\geq0)$  are independent of $\widetilde{\boldsymbol{X}}$ and, conditionally on $\boldsymbol{Z}$, are  mutually independent. Moreover, these three processes are described pathwise as follows.
\begin{itemize}

	\item[i)]\textbf{Continuous immigration.} The process $\boldsymbol{I}^{\mathbb{N}^\dag}$ is  $\mathcal{M}(\mathbb{R}^d)^\ell$-valued such that
	\begin{equation}
	\boldsymbol{I}_t^{\mathbb{N}^\dag}=\sum_{u\in\mathcal{T}}\sum_{t\wedge\tau_u\leq r< t\wedge\sigma_u}\boldsymbol{X}_{t-r}^{(1,u,r)},\notag
	\end{equation}
	where, given $\boldsymbol{Z}$, independently for each $u\in\mathcal{T}$ such that $\tau_u<t$, the processes $\boldsymbol{X}^{(1,u,r)}$ are countable in number and correspond to $\mathcal{X}$-valued
	Poissonian immigration along the space-time trajectory $\{(z_u(r),r),r\in[\tau_u,t\wedge\sigma_u)\}$ with rate $2\beta_{\gamma_u} \mathrm{d}r\times \mathrm{d}\mathbb{N}^\dag_{z_u(r)\boldsymbol{e}_{\gamma_u}}$.
	
	\item[ii)]\textbf{Discontinuous immigration.} The process $\boldsymbol{I}^{\mathbb{P}^\dag}$ is $\mathcal{M}(\mathbb{R}^d)^\ell$-valued  such that
	\begin{equation}
	\boldsymbol{I}_t^{\mathbb{P}^\dag}=\sum_{u\in\mathcal{T}}\sum_{t\wedge \tau_u\leq r< t\wedge\sigma_u}\boldsymbol{X}_{t-r}^{(2,u,r)}\notag
	\end{equation}
	where, given $\boldsymbol{Z}$, independently for each $u\in\mathcal{T}$ such that $\tau_u\leq t$, the processes $\boldsymbol{X}_{\cdot}^{(2,u,r)}$ are countable in number and 
	correspond to $\mathcal{X}$-valued, Poissonian immigration along the space-time trajectory $\{(z_u(r),r),r\in[\tau_u,t\wedge\sigma_u)\}$ with rate 
	\[\mathrm{d}r\times\int_{\boldsymbol{y}\in\mathbb{R}_+^\ell}y_{\gamma_u}\mathrm{e}^{-[ \boldsymbol{w},\boldsymbol{y}] }\Pi({\gamma_u}, \mathrm{d}\boldsymbol{y})\times \mathrm{d}\mathbb{P}^\dag_{\boldsymbol{y}\delta_{z_u(r)}}.\]

	\item[iii)]\textbf{Branch point based immigration.} The process $\boldsymbol{I}^{\eta}$ is $\mathcal{M}(\mathbb{R}^d)^\ell$-valued  such that
	\begin{equation}
	\boldsymbol{I}_t^{\eta}=\sum_{u\in\mathcal{T}}\boldsymbol{1}_{\{\sigma_u\leq t\}}\boldsymbol{X}_{t-\sigma_u}^{(3,u)}\notag
	\end{equation}
	where, given $\boldsymbol{Z}$, independently for each $u\in\mathcal{T}$ such that $\sigma_u\leq t$, the process $\boldsymbol{X}_{\cdot}^{(3,u)}$ is an independent copy of $\boldsymbol{X}$ issued at time $\sigma_u$ with law $\mathbb{P}_{\boldsymbol{Y}_u\delta_{z_u(\sigma_u)}}$ where $\boldsymbol{Y}_u$  is an independent random variable with distribution $\eta^{(\gamma_u)}_{\mathcal{N}_1^u,\dots,\mathcal{N}_\ell^u}(\mathrm{d}\boldsymbol{y})$.
	Here $(\mathcal{N}_1^u,\dots,\mathcal{N}_\ell^u)$ is the offspring of $u$,  i.e. $\mathcal{N}_i^u$ is the number of offspring of type $i$.
	
\end{itemize}
Moreover, we denote the law of the pair $(\boldsymbol{\Lambda},\boldsymbol{Z})$ by $\widehat{\mathbb{P}}_{(\boldsymbol{\mu},\boldsymbol{\nu})}$.
\end{defi}
Since $\boldsymbol{Z}$ is a MBDP and, given $\boldsymbol{Z}$,   immigrating mass  occurs independently according to a Poisson point process  or at the splitting times of $\boldsymbol{Z}$, we can deduce that  the process $((\boldsymbol{\Lambda},\boldsymbol{Z}),\widehat{\mathbb{P}}_{(\boldsymbol{\mu},\boldsymbol{\nu})})$ is Markovian. It is important to note that the  mass  which has immigrated up to  a fixed time evolves in a Markovian way thanks to the branching property.

Now we are ready to state the main results of the paper. Our first result determines the law of the couple $(\boldsymbol{\Lambda},\boldsymbol{Z})$,  and in particular shows that $\boldsymbol{\Lambda}$ is conservative.

\begin{theorem}\label{thm1}
For $\boldsymbol{\mu}\in\mathcal{M}(\mathbb{R}^d)^\ell$, $\boldsymbol{\nu}\in\mathcal{M}_a(\mathbb{R}^d)^\ell$, $\boldsymbol{f},\boldsymbol{h}\in \mathcal{B}^+(\mathbb{R}^d)^\ell$, and  $t\geq 0$ we have 
\begin{equation}\label{laplceexpjoint}
\widehat{\mathbb{E}}_{(\boldsymbol{\mu},\boldsymbol{\nu})}\left[\mathrm{e}^{-\langle \boldsymbol{f},\boldsymbol{\Lambda}_t\rangle-\langle \boldsymbol{h},\boldsymbol{Z}_t\rangle}\right]=\exp\left\{-\langle \boldsymbol{V}_t^\dag\boldsymbol{f},\boldsymbol{\mu}\rangle-\langle \boldsymbol{U}^{(\boldsymbol{f})}_t\boldsymbol{h},\boldsymbol{\nu}\rangle \right\},
\end{equation}
	where   $\boldsymbol{V}^\dag$ is defined  in \eqref{semi-cond},  and $\exp\{-\boldsymbol{U}^{(\boldsymbol{f})}_t\boldsymbol{h}(x)\}=(\exp\{-{U}^{(\boldsymbol{f}, 1)}_t\boldsymbol{h}(x)\},\cdots,\exp\{-{U}^{(\boldsymbol{f}, \ell)}_t\boldsymbol{h}(x)\})^{\tr}: \R^d\rightarrow\R_+^\ell$ 
	is the unique $[0,1]^\ell$-valued solution to the system of integral equations
\begin{equation}\label{cthm1}
\mathrm{e}^{-{U}^{(\boldsymbol{f},i)}_t\boldsymbol{h}(x)}=\mathtt{P}^{(i)}_t\mathrm{e}^{-h_i(x)}+\frac{1}{w_i}\int_0^t \mathrm{d}s \int_{\R^d}   \left[\psi^\dag\left(i, -\boldsymbol{w}\cdot \mathrm{e}^{-\boldsymbol{U}^{(\boldsymbol{f})}_{t-s}\boldsymbol{h}(y)}
+\boldsymbol{V}^{\dag}_{t-s}\boldsymbol{f}(y)\right)-\psi^\dag(i,\boldsymbol{V}^{\dag}_{t-s}\boldsymbol{f}(y))\right] \mathtt{P}^{(i)}_s(x,\ud y)
\end{equation}
for $x\in \mathbb{R}^d$, and $t\geq 0$. In particular, for each $t\ge 0$, $\boldsymbol{\Lambda}_t$ has almost surely finite mass.
\end{theorem}

Finally, we state the main result of this paper which, actually, is a consequence of Theorem \ref{thm1}. To be more precise, we consider a randomised  version of the law  $\mathbb{P}_{(\boldsymbol{\nu},\boldsymbol{\mu})}$  by replacing the deterministic choice of $\boldsymbol{\nu}$ in such a way that for each $i\in S$, $\nu_i$ is a Poisson random measure in $\mathbb{R}^d$ having intensity $w_i\mu_i$. The resulting law is denoted by $\widehat{\mathbb{P}}_{\boldsymbol{\mu}}$.

\begin{theorem}\label{mainthm}
For any $\boldsymbol{\mu}\in\mathcal{M}(\mathbb{R}^d)^\ell$ the process $(\boldsymbol{\Lambda},\widehat{\mathbb{P}}_{\boldsymbol{\mu}})$ is Markovian and has the same law as $(\boldsymbol{X}, \mathbb{P}_{\boldsymbol{\mu}})$.
\end{theorem}

The remainder of this paper is devoted to the proofs of all the results presented in the Introduction.

\section{Proofs}\label{proofs}

We first present the proofs of Propositions  \ref{Proposition1},\ref{finite extinction} and \ref{Dyn-Kuz} which are devoted to the construction of the multitype superprocess  $\boldsymbol{X}$ and its associated Dynkin-Kuznetsov measures.

\begin{proof}[Proof of Proposition \ref{Proposition1}]
Recall that  $(\mathtt{P}_t^{(i)}, t\geq 0)$ denotes  the semigroup of the diffusion $(\xi^{(i)}_t, t\geq 0)$.  
 We introduce   $\Xi=(\Xi_t, t\ge 0)$  a Markov process in the product space $\R^d\times S$ whose transition semigroup $(\mathtt{T}_t,t\geq 0)$  is given  by
 \begin{equation}
 \label{ttdefinition}
 {\mathtt{T}}_tf(x,i)=\int_{\R^d}f(y,i)\mathtt{P}_t^{(i)}(x,\ud y) \qquad \textrm{for}\quad x\in \R^d,
  \end{equation}
where $f$ is a  bounded Borel function on $\R^d\times S$.  We denote the aforementioned set of functions by $\mathcal{B}( \R^d\times S)$ and we use  $\mathcal{M}(\R^d\times S)$ for the space of finite Borel measures on $\R^d\times S$, endowed with the topology of weak convergence.

For each ${f}\in \mathcal{B}( \R^d\times S)$, we introduce the operator
 $${\Psi}(x,i,f)=\psi(i,(f(x,1),\cdots,f(x,\ell))).$$
Recall that  for a measure $\mu\in \mathcal{M}(\R^d\times S)$,  we use the notation  
\[
\langle f,\mu\rangle=\int_{\R^d\times S}f(x,i)\mu(\ud(x,i)).
\]
 Following the theory developed in the monograph of Li \cite{Z}, we observe that the operator $\Psi$ satisfies equation (2.26) in \cite{Z},  and that the assumptions of  Theorems 2.21 and 5.6, in the same monograph,  are fulfilled.  Therefore there exits a strong Markov superprocess $\mathcal{Z}=(\mathcal{Z}_t,\mathcal{G}_t, \mathbb{Q}_{\mu})$  with state space $\mathcal{M}(\mathbb{R}^d\times S)$,  and transition probabilities determined by
 \begin{align*}
 \mathbb{Q}_{\mu}\left[\mathrm{e}^{-\langle f,\mathcal{Z}_t\rangle}\right]=\exp\Big\{- \langle \mathtt{V}_tf,\mu\rangle\Big\},\qquad t\geq 0,
 \end{align*}
 where $ f\in \mathcal{B}( \R^d\times S)$ and  $t\mapsto \mathtt{V}_tf$ is the unique locally bounded positive solution to

 \begin{equation*}
\mathtt{V}_tf(x, i)=\mathtt{T}_t f(x,i)-\int_0^t\ud s\int_{\R^d\times S}\Psi(y,j,\mathtt{V}_{t-s}{f})\mathtt{T}_s(x,i,\ud (y,j)).
 \end{equation*}

For $i\in S$ and $\mu\in \mathcal{M}(\mathbb{R}^d\times S)$,  we define ${\tt U}_i\mu\in  \mathcal{M}(\mathbb{R}^d)$ by ${\tt U}_i\mu(B)=\mu(B\times\{i\})$ for  $B\in \mathcal{B}(\R^d)$, the Borel sets in $\R^d$. Observe that $\mu\mapsto ({\tt U}_i\mu)_{i\in S}$  is a homeomorphism between $\mathcal{M}(\mathbb{R}^d\times S)$ and $\mathcal{M}(\mathbb{R}^d)^{\ell}$. 
 In other words,  we can define a strong Markov process $\boldsymbol{X}\in \mathcal{M}(\R^d)^\ell$ associated with  $\mathcal{Z}$ and $(U_i)_{i\in S}$  as follows. For each $i\in S$, we define  $X_t(i,\ud x):={\tt U}_i\mathcal{Z}_t(\ud x)=\mathcal{Z}_t(\ud x\times \{i\})$ with  probabilities $\mathbb{P}_{\boldsymbol{ \mu}}:=\mathbb{Q}_{\mu}$, where  $\boldsymbol{\mu}=(\mu_1,\cdots,\mu_{\ell})\in\mathcal{M}(\R^d)^\ell$,  and each $\mu_i={\tt U}_i\mu$. In a similar way, there is a  homeomorphism between $\mathcal{B}(\R^d)^\ell$ and  $\mathcal{B}(\R^d\times S)$; that is to say for  $\boldsymbol{f}\in\mathcal{B}(\R^d)^\ell$ we define $f(x,i)=f_i(x)$. By applying the aforementioned homeomorphisms, we deduce that $(\boldsymbol{X}_t,\mathbb{P}_{\boldsymbol{ \mu}})$ satisfies
 \eqref{laplace},  and \eqref{inteq_u} has a unique locally bounded solution.
 \end{proof}

We now prove Proposition \ref{finite extinction}, which will be very useful for  the existence of Dynkin-Kutznetsov measures.
\begin{proof}[Proof of  Proposition \ref{finite extinction}]
	 By \eqref{w} and the branching property of $\boldsymbol{X}$ we have 
	\begin{equation}
	\label{extinction branching}
	\mathbb{P}_{\boldsymbol{\mu}}(\mathcal{E})=\mathrm{e}^{-\langle \boldsymbol{w},\boldsymbol{\mu}\rangle}.
	\end{equation}
	Furthermore by conditioning the event $\mathcal{E}$ on  $\mathcal{F}_t$ and using the Markov property,  we obtain that 
	\begin{equation*}
	\begin{split}
	\mathrm{e}^{-\langle \boldsymbol{w},\boldsymbol{\mu}\rangle}&=\mathbb{E}_{\boldsymbol{\mu}}\Big[\mathbb{E}[\boldsymbol{1}_\mathcal{E}|\mathcal{F}_t]\Big]=\mathbb{E}_{\boldsymbol{\mu}}\Big[\mathbb{E}_{\boldsymbol{X}_t}[\boldsymbol{1}_\mathcal{E}]\Big]=\mathbb{E}_{\boldsymbol{\mu}}\left[ \mathrm{e}^{-\langle \boldsymbol{w},\boldsymbol{X}_t\rangle}\right].
	\end{split}
	\end{equation*}
	Thus  from \eqref{inteq_u} and the assumption \eqref{extinction assumption} we also get that $\boldsymbol{\psi}(\boldsymbol{w})=\boldsymbol{0}$.
	 
For the second part of the statement, we recall the definition of the total mass vector $\boldsymbol{Y}=(\boldsymbol{Y}_t,t\geq 0)$  whose entries satisfy
$Y_t(i)=X_t(i,\R^d)$. From identity \eqref{ec mcsb} and assumption \eqref{extinction assumption}, we know that for each $i\in S$, there exists a positive deterministic time $T_i$ such that
	\[
	\mathbf{P}_{\boldsymbol{e}_i}(\|\boldsymbol{Y}_t\|=0 )=\mathrm{e}^{-\lim_{\boldsymbol{\theta}\hookrightarrow\infty} v_t(i,\boldsymbol{\theta})}\left\{
	\begin{array}{ll}
	=0 & \textrm{ for } t< T_i,\\
	>0 &\textrm{ for } t>T_i,
	\end{array}\right.
	\]
	where $\boldsymbol{v}_t(i,\boldsymbol{\theta})$ is  given by \eqref{ec v for mcsb},  and we recall that $\boldsymbol{\theta}\hookrightarrow\infty$ means that each coordinate of $\boldsymbol{\theta}$ goes to $\infty$. 

Next, we define the sets $S_1:=\{i\in S: T_i=0\}$ and $S_2:=\{i\in S: T_i>0\}$. For a vector $\boldsymbol{y}=(y_1,\cdots,y_\ell)$, we denote its support by $\textrm{supp}(\boldsymbol{y}):=\{i\in S: y_i\neq 0\}$. 
	Thus, the  proof will be  completed if we show that $S_2=\emptyset.$ We proceed by contradiction.
	
 Let us assume that $S_2\neq\emptyset$ and define $T:=\inf\{T_i: i\in S_2\}$ which is strictly positive by definition. 
Take $i\in S_2$ and observe from  the Markov property that 
\[
\begin{split}
0&=\mathbf{P}_{\boldsymbol{e}_i}\Big(\|\boldsymbol{Y}_{3T/4}\|=0 \Big)\\
&\geq \mathbf{P}_{\boldsymbol{e}_i}\Big(\|\boldsymbol{Y}_{3T/4}\|=0, \textrm{supp}(\boldsymbol{Y}_{T/2}) \subset S_1\Big)\\
&=\mathbf{E}_{\boldsymbol{e}_i}\left[\mathbf{P}_{\boldsymbol{Y}_{T/2}}\Big(\|\boldsymbol{Y}_{T/4}\|=0\Big),  \textrm{supp}(\boldsymbol{Y}_{T/2}) \subset S_1 \right].
\end{split}
\]
By the branching property, if $\boldsymbol{y}$ is a vector such that $\textrm{supp}(\boldsymbol{y})\subset S_1$ then $\mathbf{P}_{\boldsymbol{y}}(\|\boldsymbol{Y}_t\|=0 )>0$, for all $t>0$. Therefore, we necessarily have 
 \[
 0=\mathbf{P}_{\boldsymbol{e}_i}\Big( \textrm{supp}(\boldsymbol{Y}_{T/2})\subset S_1\Big),
 \]
  and implicitly
$$1=\mathbf{P}_{\boldsymbol{e}_i}\Big(\textrm{supp}(\boldsymbol{Y}_{T/2}) \cap S_2\neq \emptyset \Big)=\mathbf{P}_{\boldsymbol{e}_i}\Big(\|\boldsymbol{Y}_{T/2}\|>0 \Big), \qquad \mbox{for all } i\in S_2.$$
Hence, using the branching property again,  if $\boldsymbol{y}$ is a vector such that $\textrm{supp}(\boldsymbol{y})\cap S_2\neq \emptyset$, we have 
$$1=\mathbf{P}_{\boldsymbol{y}}\Big(\|\boldsymbol{Y}_{T/2}\|>0 \Big)=\mathbf{P}_{y}\Big(\textrm{supp}(\boldsymbol{Y}_{T/2}) \cap S_2\neq \emptyset \Big).$$
Finally, we use the  Markov property recursively and the previous equality, to deduce  that for all $k\ge 1$,
$$\mathbf{P}_{\boldsymbol{y}}\Big(\|\boldsymbol{Y}_{kT/2}\|>0 \Big)=1  \qquad \mbox{for all } i\in S_2,$$
which is inconsistent  with the definitions of $T$ and $T_i$. In other words, $S_2=\emptyset$. This completes the proof.
\end{proof}

We now prove the existence of the Dynkin-Kuznetsov measures.
\begin{proof}[Proof of  Proposition \ref{Dyn-Kuz}]
Let us denote by $\mathcal{M}^0(\mathbb{R}^d\times S):=\mathcal{M}(\mathbb{R}^d\times S)\setminus \{ 0\}$,  where $0$ is the null measure. Consider the Markov superprocess $\mathcal{Z}$ introduced in the previous proof. Let  $({\tt Q}_t, t\geq 0)$ and $({\tt V}_t, t\geq 0)$ be the transition and cumulant semigroups associated with $\mathcal{Z}$.  
By Theorem 1.36 in \cite{Z}, $\mathtt{V}_t$ has the following representation 
\[
\mathtt{V}_t f(x,i)=\int_{\mathbb{R}^d\times S} f(y,j)\Lambda_t(x,i,\mathrm{d}(y,j))+\int_{\mathcal{M}^0(\mathbb{R}^d\times S)}\left( 1-\mathrm{e}^{-\langle f,\nu\rangle}\right)L_t(x,i,\mathrm{d}\nu), \quad t\geq 0,
\]
where  $f$ is a positive Borel function on $\mathbb{R}^d\times S$, $\Lambda_t(x,i,\mathrm{d}(y,j))$ is a bounded kernel on $\mathbb{R}^d\times S$, and $(1\wedge \langle 1,\nu\rangle)L_t(x,i,\mathrm{d}\nu)$ is a bounded kernel from $\mathbb{R}^d\times S$ to $\mathcal{M}^0(\mathbb{R}^d\times S)$.

Let $\widetilde{\mathcal{X}}^+$ be the space of c\`adl\`ag paths $t\rightarrow \widetilde{w}_t$ from $[0,\infty)$ to $\mathcal{M}(\mathbb{R}^d\times S)$ having the null measure as a trap.  Let $({\tt Q}_t^0, t\ge 0)$ be the restriction of $({\tt Q}_t, t\ge 0)$ to $\mathcal{M}^0(\mathbb{R}^d\times S)$ and 
\[
E_0:=\Big\{ (x,i)\in \mathbb{R}^d\times S:\Lambda_t(x,i,\mathbb{R}^d\times S)=0,\; \textrm{ for all }\,\, t>0\Big\}.
\]
By Proposition 2.8 in \cite{Z},  for all $(x,i)\in E_0$ the family of measures $(L_t(x,i,\cdot), t\geq 0)$ on $\mathcal{M}^0(\mathbb{R}^d\times S)$ constitutes an entrance law for $({\tt Q}_t^0, t\geq 0)$.
Therefore, by Theorem A.40 of \cite{Z} for all $(x,i)\in E_0$ there exists a unique $\sigma$-finite measure $\widetilde{\mathbb{N}}_{(x,i)}$ on $\widetilde{\mathcal{X}}^+$ such that $\widetilde{\mathbb{N}}_{(x,i)}(\{ 0\})=0$,  and for any $0<t_1<\cdots<t_n<\infty$
\[
\widetilde{\mathbb{N}}_{(x,i)}(\mathcal{Z}_{t_1}\in\mathrm{d}\nu_1,\mathcal{Z}_{t_2}\in\mathrm{d}\nu_2,\dots,\mathcal{Z}_{t_n}\in\mathrm{d}\nu_n)=L_{t_1}(x,i,\mathrm{d}\nu_1){\tt Q}_{t_2-t_1}^0(\nu_1,\mathrm{d}\nu_2)\dots {\tt Q}_{t_n-t_{n-1}}^0 (\nu_{n-1},\mathrm{d}\nu_n).
\]
It follows that for all $t>0$, $(x,i)\in E_0$,  and $f \in\mathcal{B}(\mathbb{R}^d\times S)$ positive, we have
\[
\widetilde{\mathbb{N}}_{(x,i)}\left( 1-\mathrm{e}^{-\langle f,\mathcal{Z}_t\rangle}\right)=\int_{\mathcal{M}^0(\mathbb{R}^d\times S)}\left( 1-\mathrm{e}^{-\langle f,\nu\rangle}\right)L_t(x,i,\mathrm{d}\nu)=\mathtt{V}_tf(x,i).
\]

Recall the homeomorphism $\mu\mapsto (U_i\mu)_{i\in S}$   and the definition of the superprocess $\boldsymbol{X}$ from the proof of Proposition \ref{Proposition1}. By taking the constant function $f(x,i)=\lambda\in \R$, and  using the definitions of $\mathtt{V}_t, {\tt Q}_t$,  we deduce that 
$$-\log\mathbb{E}_{\boldsymbol{e}_i \delta_x}\Big[\mathrm{e}^{-\lambda\langle \boldsymbol{1},\boldsymbol{X}_t\rangle}\Big]=\lambda  \Lambda_t(x,i,\mathbb{R}^d\times S)+\int_{\mathcal{M}^0(\mathbb{R}^d\times S)}\left( 1-\mathrm{e}^{-\lambda\langle 1 ,\nu\rangle}\right)L_t(x,i,\mathrm{d}\nu).$$
 If we take  $\lambda$ goes to infinity, the left hand side of the above identity converges to $-\log \mathbb{P}_{\boldsymbol{e}_i \delta_x}(\|\boldsymbol{X}_t\|=0)$ which is finite by Proposition \eqref{finite extinction}. Henceforth, $\Lambda_t(x,i,\mathbb{R}^d\times S)=0$ and $(x,i)\in E_0$. 

Next, recall that $\mathcal{X}$ denotes the space  of c\`adl\`ag paths  from $[0,\infty)$ to $\mathcal{M}(\mathbb{R}^d)^\ell$. Then $({\tt U}_i)_{i\in S}$  induces an homeomorphism   between $\widetilde{\mathcal{X}}$ and $\mathcal{X}$. More precisely, the homeomorphism  $\mathcal{U}:\widetilde{\mathcal{X}}\rightarrow\mathcal{X}$   is given by $\widetilde{w}_t\rightarrow \boldsymbol{w}_t=(w_t(1),\cdots, w_t(\ell))$ where for all $i\in S$ the measure in the $i$th coordinate is given by  $w_t(i,B)=\widetilde{w}_t(B\times\{i\})$.  
This implies that for all $(x,i)\in \mathbb{R}^d\times S$ we can define the measures $\mathbb{N}_{x\boldsymbol{e}_i}$  on $\mathcal{X}$ given by $\mathbb{N}_{x\boldsymbol{e}_i}(B):= \widetilde{\mathbb{N}}_{(x,i)}(\mathcal{U}^{-1}(B))$.  In other words, we obtain 
\begin{equation*}
\mathbb{N}_{x\boldsymbol{e}_i}\left(1-\mathrm{e}^{-\langle \boldsymbol{f},\boldsymbol{X}_t\rangle}\right)=-\log\mathbb{E}_{\boldsymbol{e}_i \delta_x}\left[\mathrm{e}^{-\langle \boldsymbol{f},\boldsymbol{X}_t\rangle}\right],
\end{equation*}
for all $\boldsymbol{f}\in \mathcal{B}(\R^d)^\ell$ and $t\geq0$.

\end{proof}

It is important to note that   the Dynkin-Kuznetsov measures $\mathbb{N}^\dag$  associated to the multitype superprocess conditioned on extinction  are  also well defined since $|\log \mathbb{P}^{\dag}_{\delta_x\boldsymbol{e}_i}(\mathcal{E})|<\infty$.

We now prove Proposition  \ref{proposition2}.
\begin{proof}[Proof of Proposition \ref{proposition2}]
	Using \eqref{extinction branching}, \eqref{extinction assumption} and the Markov property, we have for $\boldsymbol{f}\in\mathcal{B}^+(\mathbb{R}^d)^\ell$ 
	\begin{align*}
	\mathbb{E}_{\boldsymbol{\mu}}^\dag\left[\mathrm{e}^{-\langle \boldsymbol{f},\boldsymbol{X}_t\rangle}\right]& 
	=\mathrm{e}^{\langle \boldsymbol{w},\boldsymbol{\mu}\rangle}\mathbb{E}_{\boldsymbol{\mu}}\left[\mathrm{e}^{-\langle \boldsymbol{f},\boldsymbol{X}_t\rangle}\boldsymbol{1}_{\mathcal{E}}\right]\\
	&=\mathrm{e}^{\langle \boldsymbol{w},\boldsymbol{\mu}\rangle}\mathbb{E}_{\boldsymbol{\mu}}\left[\mathrm{e}^{-\langle \boldsymbol{f},\boldsymbol{X}_t\rangle}\mathbb{P}_{\boldsymbol{X}_t}(\mathcal{E})\right]\\
	&=\mathrm{e}^{\langle \boldsymbol{w},\boldsymbol{\mu}\rangle}\mathbb{E}_{\boldsymbol{\mu}}\left[\mathrm{e}^{-\langle \boldsymbol{f},\boldsymbol{X}_t\rangle}\mathrm{e}^{-\langle \boldsymbol{w},\boldsymbol{X}_t\rangle}\right]\\
	&=\mathrm{e}^{-\langle \boldsymbol{V}_t (\boldsymbol{f}+\boldsymbol{w})-\boldsymbol{w},\boldsymbol{\mu}\rangle}.
	\end{align*}
Since $ \boldsymbol{V}_t (\boldsymbol{f}+\boldsymbol{w})$ satisfies \eqref{inteq_u}, using  the definitions of $\boldsymbol{V}^{\dag}_t \boldsymbol{f}$ and $\boldsymbol{\psi}^\dag$ we obtain that $\boldsymbol{V}^{\dag}_t \boldsymbol{f}$  satisfies \eqref{semi-cond}. Recalling that  $\boldsymbol{\psi}(\boldsymbol{w})=\boldsymbol{0}$ and computing $\boldsymbol{\psi}(\boldsymbol{\theta}+\boldsymbol{w})-\boldsymbol{\psi}(\boldsymbol{w})$, we deduce that 
	\begin{equation}\label{psistar}
	\psi^{\dag}(i,\boldsymbol{\theta})=-[ \boldsymbol{\theta},{\boldsymbol{B}}^\dag\boldsymbol{e}_i]+\beta_i \theta_i^2+\int_{\mathbb{R}_+^\ell}\left( \mathrm{e}^{-[ \boldsymbol{\theta},\boldsymbol{y}]}-1+\theta_i y_i\right)\mathrm{e}^{-[ \boldsymbol{w},\boldsymbol{y}]}\Pi(i,\ud\boldsymbol{y}),
	\end{equation}
	where 
	\begin{equation}\label{Bdag}
	{B}_{ij}^\dag=B_{ij}-\left(2\beta_i w_i +\int_{\mathbb{R}_+^\ell}\left( 1-\mathrm{e}^{-[ \boldsymbol{y},\boldsymbol{w}]} \right)y_i\Pi(i,\mathrm{d}\boldsymbol{y})\right)\boldsymbol{1}_{\{j=i\}}.
	\end{equation}
This implies that $\boldsymbol{\psi}^\dag$ is a branching mechanism and therefore the solution of \eqref{semi-cond} is unique. In other words,  $\boldsymbol{X}$ under $\mathbb{P}^\dag_{\boldsymbol{\mu}}$ is a multitype superprocess with branching mechanism given by $\boldsymbol{\psi}^\dag(\boldsymbol{\theta})$.
\end{proof}

In order to proceed with the proof of Theorem \ref{thm1},  the following two lemmas are necessary. 
\begin{lemma}\label{lemma1}
	For each $\boldsymbol{f}\in \mathcal{B}(\mathbb{R}^d)^\ell$, $\boldsymbol{\nu}\in\mathcal{M}_a(\mathbb{R}^d)^\ell$,  $\boldsymbol{\mu}\in\mathcal{M}(\mathbb{R}^d)^\ell$, and $t\geq0$ we have
	\begin{equation}
	\widehat{\mathbb{E}}_{(\boldsymbol{\mu},\boldsymbol{\nu})}\left[\left.\mathrm{e}^{-\langle \boldsymbol{f}, \boldsymbol{I}_t^{\mathbb{N}^\dag}+\boldsymbol{I}_t^{\mathbb{P}^\dag}\rangle} \right| (\boldsymbol{Z}_s,s\leq t)\right]=\exp\left\{-\int_0^t \langle \boldsymbol{\phi}(\boldsymbol{V}^\dag_{t-r}\boldsymbol{f}), \boldsymbol{Z}_r\rangle \mathrm{d}r\right\},
	\end{equation}
	where $\boldsymbol{\phi}$ is given by \eqref{phi} and $\boldsymbol{V}^\dag_{t}\boldsymbol{f}$ satisfies \eqref{semi-cond}.
\end{lemma}

\begin{proof}
As the different immigration mechanisms are independent given the backbone, we may look at the Laplace functional of the continuous and discontinuous immigrations  separately.
For the continuous immigration, we can condition on $\boldsymbol{Z}$, use Campbell's formula, then equation  (\ref{NmeasureLaplace})  for $\mathbb{N}^\dag$, and finally the definition of $\boldsymbol{V}^\dag_{t}\boldsymbol{f}(x)=({V}_{t}^{\dag,(1)}\boldsymbol{f}(x),\cdots,{V}_{t}^{\dag, (\ell)}\boldsymbol{f}(x))^{\tt t}$ to obtain 
\begin{equation*}
\begin{split}
\widehat{\mathbb{E}}_{(\boldsymbol{\mu},\boldsymbol{\nu})}\left[ \exp\{-\left.\langle\boldsymbol{ f}, \boldsymbol{I}_t^{\mathbb{N}^\dag}\rangle\}\right| (\boldsymbol{Z}_{s},s\leq t)\right]&=\exp\left\lbrace -\sum_{u\in\mathcal{T}} 2\beta_{\gamma_u}\int_{t\wedge \tau_u}^{t\wedge\sigma_u}\mathrm{d}r \mathbb{N}_{z_u(r)\boldsymbol{e}_{\gamma_u}}^\dag \Big(1-\mathrm{e}^{-\langle \boldsymbol{f}, \boldsymbol{X}_{t-r}\rangle}\Big)\right\rbrace\\
&=\exp\left\lbrace -\sum_{u\in\mathcal{T}}2\beta_{\gamma_u}\int_{t\wedge \tau_u}^{t\wedge \sigma_u}\mathrm{d}r V_{t-r}^{\dag, (\gamma_u)} \boldsymbol{f}(z_u(r)) \right\rbrace.
\end{split}
\end{equation*}
In a similar way, for the  discontinuous immigration, by conditioning on $\boldsymbol{Z}$,  using Campbell's formula and the definition of $\boldsymbol{V}^\dag_{t}\boldsymbol{f}$  we get
\begin{equation*}
\begin{split}
\widehat{\mathbb{E}}_{(\boldsymbol{\mu},\boldsymbol{\nu})}&\left[ \left.\exp\{-\langle \boldsymbol{f}, \boldsymbol{I}_t^{\mathbb{P}^\dag}\rangle\}\right| (\boldsymbol{Z}_{s},s\leq t)\right]\\
&=
\exp\left\lbrace -\sum_{u\in\mathcal{T}}\int_{t\wedge\tau_u}^{t\wedge \sigma_u}\mathrm{d}r \int_{\mathbb{R}_+^\ell} y_{\gamma_u} \mathrm{e}^{-[ \boldsymbol{w},\boldsymbol{y}]}  \mathbb{E}_{\boldsymbol{y}\delta_{z_u(r)}}^\dag \left[ 1-\mathrm{e}^{-\langle\boldsymbol{f},\boldsymbol{X}_{t-r}\rangle}\right]\Pi(\gamma_u,\mathrm{d}\boldsymbol{y})\right\rbrace\\
&=\exp\left\lbrace -\sum_{u\in\mathcal{T}}\int_{t\wedge\tau_u}^{t\wedge\sigma_u} \mathrm{d}r \int_{\mathbb{R}_+^\ell} y_{\gamma_u}  \mathrm{e}^{-[ \boldsymbol{w},\boldsymbol{y}]} \left( 1-\mathrm{e}^{-[\boldsymbol{V}_{t-r}^\dag\boldsymbol{f}(z_u(r)),\boldsymbol{y}]}\right)\Pi(\gamma_u,\mathrm{d}\boldsymbol{y})  \right\rbrace.
\end{split}
\end{equation*}
Therefore, by putting the pieces together we obtain the following
\begin{equation}\label{eq:contdiscont}
\widehat{\mathbb{E}}_{(\boldsymbol{\mu},\boldsymbol{\nu})}\left[ \left.\exp\left\lbrace -\langle \boldsymbol{f}, \boldsymbol{I}_t^{\mathbb{N}^\dag}+\boldsymbol{I}_t^{\mathbb{P}^\dag}\rangle \right\rbrace  \right| (\boldsymbol{Z}_s,s\leq t)\right]=\exp\left\lbrace -\sum_{u\in\mathcal{T}}\int_{t\wedge\tau_u}^{t\wedge \sigma_u}\phi(\gamma_u,
\boldsymbol{V}_{t-r}^\dag\boldsymbol{f}(z_u(r)))\mathrm{d}r\right\rbrace,
\end{equation}
where $\phi(i,\boldsymbol{\lambda})$  is given by formula \eqref{phi}. The previous equation is in terms of the tree $\mathcal{T}$. We want to rewrite it in terms of the multitype branching diffusion, thus 
\begin{equation*}
\begin{split}
\sum_{u\in\mathcal{T}}\int_{t\wedge\tau_u}^{t\wedge \sigma_u}\phi(\gamma_u,
\boldsymbol{V}_{t-r}^\dag\boldsymbol{f}(z_u(r))) \mathrm{d}r=&\sum_{i\in S}\sum_{u\in\mathcal{T},\gamma_u=i}\int_{t\wedge\tau_u}^{t\wedge \sigma_u}\phi(i,
\boldsymbol{V}_{t-r}^\dag\boldsymbol{f}(z_u(r))) \mathrm{d}r\\
=&\int_0^t \sum_{i\in S}\sum_{u\in\mathcal{T},z_u=i}\phi(i,
\boldsymbol{V}_{t-r}^\dag\boldsymbol{f}(z_u(r))) \boldsymbol{1}_{\{r\in [t\wedge \tau_u,t\wedge \sigma_u)\}}\mathrm{d}r\\
 =& \int_0^t \left\langle \boldsymbol{\phi}(\boldsymbol{V}^\dag_{t-r}), \boldsymbol{Z}_r\right\rangle \mathrm{d}r.
\end{split}
\end{equation*}
\end{proof}
Observe that the processes $\boldsymbol{I}^{\mathbb{N}^\dag}= (\boldsymbol{I}^{\mathbb{N}^\dag}_t,t\geq0)$, $\boldsymbol{I}^{\mathbb{P}^\dag}=(\boldsymbol{I}^{\mathbb{P}^\dag}_t,t\geq0)$ and $\boldsymbol{I}^{\eta}=(\boldsymbol{I}^{\eta}_t,t\geq0)$ are initially zero-valued  $\widehat{\mathbb{P}}_{(\boldsymbol{\mu},\boldsymbol{\nu})}$-a.s. 
In order to study the rest of the immigration along the backbone we have the following result.
\begin{lemma}\label{lemma2}
	Suppose that $\boldsymbol{f},\boldsymbol{h}\in \mathcal{B}(\mathbb{R}^d)^\ell$ and $\boldsymbol{g}_s(x)\in\mathcal{B}(\R\times\mathbb{R}^d)^\ell$.
Define the vectorial function $\mathrm{e}^{- \boldsymbol{W}_t(x)}=(\mathrm{e}^{- {W}^{(1)}_t(x)},\cdots,\mathrm{e}^{- {W}^{(\ell)}_t(x)})$ as follows
	\begin{align*}
\mathrm{e}^{- {W}^{(i)}_t(x)}:=	\widehat{\mathbb{E}}_{(\boldsymbol{\mu},\boldsymbol{e}_i \delta_x)}\left[ \exp \left\lbrace -\langle \boldsymbol{f},\boldsymbol{I}_t^\eta\rangle-\langle \boldsymbol{h},\boldsymbol{Z}_t\rangle -\int_{0}^{t}\langle \boldsymbol{g}_{t-s},\boldsymbol{Z}_s\rangle \mathrm{d}s\right\rbrace\right].
	\end{align*}
Then, $\mathrm{e}^{- \boldsymbol{W}_t(x)}$ is a locally bounded  solution to the integral system

\begin{equation}\label{eq:inteq}
\mathrm{e}^{-W^{(i)}_t(x)}=\mathtt{P}^{(i)}_t\mathrm{e}^{-h_i(x)}+\frac{1}{w_i}\int_0^t \mathrm{d}s \int_{\R^d}   \left[H^{(i)}_{t-s}\left( y, \boldsymbol{w}\cdot \mathrm{e}^{-\boldsymbol{W}_{t-s}(y)}\right)-w_ig_{t-s}^i(y) \mathrm{e}^{-W^{(i)}_{t-s}(y)}\right] \mathtt{P}^{(i)}_s(x,\ud y),
\end{equation}
	where 
	\begin{equation}\label{H}
	\begin{split}
	H_{s}^{(i)}&(x,\boldsymbol{ \theta}) = [ \boldsymbol{\theta},{\boldsymbol{B}^{\dag}}\boldsymbol{e}_i]+\beta_i\theta_i^2+\int_{\mathbb{R}_+^\ell}\left( \mathrm{e}^{[ \boldsymbol{\theta},\boldsymbol{y}]}-1-\theta_iy_i\right)\mathrm{e}^{-[ \boldsymbol{w}+\boldsymbol{V}^\dag_s\boldsymbol{f}(x),\boldsymbol{y}]}\Pi(i,\mathrm{d}\boldsymbol{y}).
	\end{split}
	\end{equation}
	In the latter formula  $\boldsymbol{B}^\dag$ is given by \eqref{Bdag} and $\boldsymbol{V}^\dag_t\boldsymbol{f}$ is the unique solution to \eqref{semi-cond}.
\end{lemma}

It is important to note that $\boldsymbol{W}$ depends on the functions $\boldsymbol{f}, \boldsymbol{h}$ and $\boldsymbol{g}$ but for simplicity on exposition we suppress this dependency.
\begin{proof}
Recall that $\boldsymbol{Z}$ is a multitype branching diffusion, where the motion of each particle with type $i\in S$ is given by the semigroup $\mathtt{P}^{(i)}$ and its branching generator  is given by \eqref{brate}.  For simplicity, we denote by $\mathtt{P}^{(i)}_x$  the law of the diffusion $\xi^{(i)}$ starting at $x$.
By conditioning on  the time of the first branching event of $\boldsymbol{Z}$ we get
\begin{equation*}
\begin{split}
\mathrm{e}^{- {W}^{(i)}_t(x)}&=\mathtt{E}^{(i)}_x\left[\mathrm{e}^{-q_i t}\mathrm{e}^{-\int_0^tg^i_{t-r}(\xi^{(i)}_r)\mathrm{d}r}\mathrm{e}^{-h_i(\xi^{(i)}_t)} \right]\\
&+\mathtt{E}^{(i)}_x\left[\int_0^tq_i   \mathrm{e}^{-q_i s}\mathrm{e}^{-\int_0^sg^i_{s-r}(\xi^{(i)}_r)\mathrm{d}r}\sum_{\boldsymbol{j}\in\mathbb{N}^\ell}p_{\boldsymbol{j}}^{(i)} \mathrm{e}^{-\sum_{k\in S}j_k W_{t-s}^{(k)}(\xi^{(i)}_s)}\int_{\mathbb{R}_+^\ell}\eta_{\boldsymbol{j}}^{(i)}(\mathrm{d}\boldsymbol{y})\mathrm{e}^{-[  \boldsymbol{V}^{\dag}_{t-s}\boldsymbol{f}(\xi^{(i)}_t),\boldsymbol{y}]}\mathrm{d}s\right],
\end{split}
\end{equation*}
where   $\boldsymbol{j}=(j_1,\dots,j_\ell)$. On the other hand, by  Proposition 2.9 in  \cite{Z}, we see that $\mathrm{e}^{- {W}^{(i)}_t(x)}$ also satisfies
\begin{equation*}
\begin{split}
	\mathrm{e}^{- {W}^{(i)}_t(x)}=&\mathtt{E}^{(i)}_x\left[\mathrm{e}^{-h_i(\xi^{(i)}_t)} \right]-\mathtt{E}^{(i)}_x\left[\int_0^tq_i\mathrm{e}^{-{W}^{(i)}_{t-s}(x)} \ud s \right]-\mathtt{E}^{(i)}_x\left[\int_0^tg_{t-s}^i(\xi_s^{(i)})\mathrm{e}^{-{W}^{(i)}_{t-s}(x)} \ud s \right]\\
	&\hspace{2cm}+\mathtt{E}^{(i)}_x\left[\int_0^tq_i   \sum_{\boldsymbol{j}\in\mathbb{N}^\ell}p_{\boldsymbol{j}}^{(i)} \mathrm{e}^{-\sum_{k\in S}j_k W_{t-s}^{(k)}(\xi^{(i)}_s)}\int_{\mathbb{R}_+^\ell}\eta_{\boldsymbol{j}}^{(i)}(\mathrm{d}\boldsymbol{y})\mathrm{e}^{-[  \boldsymbol{V}^{\dag}_{t-s}\boldsymbol{f}(\xi^{(i)}_t),\boldsymbol{y}]}\mathrm{d}s\right].
\end{split}
\end{equation*}
 By substituting the definitions  of $p^{(i)}_{\boldsymbol{j}}$ and $\eta^{(i)}_{\boldsymbol{j}}$ (see \eqref{p^i definition} and \eqref{eq:eta}), we get that for all $x\in \R^d$ 
\begin{equation*}
\begin{split}
R(x)&:=\sum_{\boldsymbol{j}\in\mathbb{N}^\ell}p_{\boldsymbol{j}}^{(i)} \mathrm{e}^{-[ \boldsymbol{j},\boldsymbol{W}_{t-s}(x)]}\int_{\mathbb{R}_+^\ell}\eta_{\boldsymbol{j}}^i(\mathrm{d}\boldsymbol{y})
\mathrm{e}^{-[  \boldsymbol{V}^{\dag}_{t-s}\boldsymbol{f}(x) ,\boldsymbol{y}]} \\
&= 
\frac{1}{w_iq_i}\sum_{\boldsymbol{j}\in\mathbb{N}^\ell}\Bigg[ \beta_i w_i^2 \mathrm{e}^{-[ \boldsymbol{j},\boldsymbol{W}_{t-s}(x)]}\boldsymbol{1}_{\{\boldsymbol{j}=2\boldsymbol{e}_i\}} \Bigg.\\
&\hspace{.5cm}+\left(B_{ki} w_k \mathrm{e}^{-[ \boldsymbol{j},\boldsymbol{W}_{t-s}(x)]}
+\int_{\mathbb{R}_+^\ell}w_k y_k \mathrm{e}^{-[ \boldsymbol{w},\boldsymbol{y}]}\mathrm{e}^{-[ \boldsymbol{j},\boldsymbol{W}_{t-s}(x)]}\mathrm{e}^{-[ \boldsymbol{V}^{\dag}_{t-s}\boldsymbol{f}(x),\boldsymbol{y}]}\Pi(i,\mathrm{d}\boldsymbol{y})\right)\boldsymbol{1}_{\{\boldsymbol{j}=\boldsymbol{e}_k\}}\Ind{ k\neq i} \\
&\hspace{2.5cm}\left.+\int_{\mathbb{R}_+^\ell}\frac{(w_1 y_1)^{j_1}\dots(w_\ell y_\ell)^{j_\ell}}{j_1!\dots j_\ell!}\mathrm{e}^{-[ \boldsymbol{w},\boldsymbol{y}]} \mathrm{e}^{-[ \boldsymbol{j},\boldsymbol{W}_{t-s}(x)]}\mathrm{e}^{-[ \boldsymbol{V}^{\dag}_{t-s}\boldsymbol{f}(x),\boldsymbol{y}]}\Pi(i,\mathrm{d}\boldsymbol{y})\boldsymbol{1}_{\{j_1+\dots+j_\ell\geq 2\}} \right]\\
=& \frac{1}{w_iq_i}\left[ \beta_i \left(w_i \mathrm{e}^{-W^{(i)}_{t-s}(x)}\right)^2+\sum_{k\in S,k\neq i}  \mathrm{e}^{-W^{(k)}_{t-s}(x)}\left(B_{ki} w_k  + \int_{\mathbb{R}_+^\ell}w_k y_k \mathrm{e}^{-[ \boldsymbol{w},\boldsymbol{y}]}\mathrm{e}^{-[ \boldsymbol{V}^{\dag}_{t-s}\boldsymbol{f}(x),\boldsymbol{y}]}\Pi(i,\mathrm{d}\boldsymbol{y})\right)
 \right.\\
&\hspace{6.5cm}\left.+\int_{\mathbb{R}_+^\ell}\sum_{n\geq 2}\frac{[\boldsymbol{w}\cdot\mathrm{e}^{- \boldsymbol{W}_t(x)}, \boldsymbol{y}]^n}{n!}\mathrm{e}^{-[ \boldsymbol{w},\boldsymbol{y}]} \mathrm{e}^{-[ \boldsymbol{V}^{\dag}_{t-s}\boldsymbol{f}(x),\boldsymbol{y}]}\Pi(i,\mathrm{d}\boldsymbol{y})
 \right],
\end{split}
\end{equation*}
where in the last row we have used (\ref{eq:binom}).
By merging the two integrals, we get
\begin{equation*}
	\begin{split}
R(x)=&	\frac{1}{w_iq_i}\left[ \beta_i \left(w_i \mathrm{e}^{-W^{(i)}_{t-s}(x)}\right)^2+\sum_{k\in S,k\neq i} B_{ki} w_k \mathrm{e}^{-W^{(k)}_{t-s}(x)}
\right.\\
&\hspace{3cm}\left.+\int_{\mathbb{R}_+^\ell}\left( \mathrm{e}^{[\boldsymbol{w}\cdot\mathrm{e}^{- \boldsymbol{W}_t(x)}, \boldsymbol{y}]}-1-w_i\mathrm{e}^{-W^{(k)}_{t-s}(x)}y_i\right)\mathrm{e}^{-[ \boldsymbol{w}+ \boldsymbol{V}^{\dag}_{t-s}\boldsymbol{f}(x),\boldsymbol{y}]}\Pi(i,\mathrm{d}\boldsymbol{y})
\right].
	\end{split}
\end{equation*}
 So, putting the pieces together and using   the definitions  of $q_i$, $\boldsymbol{B}^\dag$ and $H^{(i)}$, (see identities \eqref{brate},\eqref{Bdag} and \eqref{H}) we deduce that 
\begin{equation*}
\begin{split}
\mathrm{e}^{- {W}^{(i)}_t(x)}=\mathtt{E}^{(i)}_x\left[\mathrm{e}^{-h_i(\xi^{(i)}_t)} -\int_0^tg_{t-s}^i(\xi_s^{(i)})\mathrm{e}^{-{W}^{(i)}_{t-s}(\xi_s^{(i)})} \ud s +\frac{1}{w_i}\int_0^t  H_{t-s}^{(i)}(\xi^{(i)}_s,\boldsymbol{w}\cdot\mathrm{e}^{- \boldsymbol{W}_t(\xi^{(i)}_s)})\mathrm{d}s\right],
\end{split}
\end{equation*}
as stated. Therefore, $\mathrm{e}^{- \boldsymbol{W}_t(x)}$ satisfies  \eqref{eq:inteq}.
\end{proof}

\begin{proof}[Proof of Theorem \ref{thm1}] 
Since $\widetilde{\boldsymbol{X}}$ is an independent copy of $\boldsymbol{X}$ under $\p_{\boldsymbol{\mu}}^\dag$, it is enough to show that for  $\boldsymbol{\mu}\in\mathcal{M}(\mathbb{R}^d)^\ell$, $\boldsymbol{\nu}\in\mathcal{M}_a(\mathbb{R}^d)^\ell$, $\boldsymbol{f},\boldsymbol{h}\in \mathcal{B}^+(\mathbb{R}^d)^\ell$, the vectorial function $\mathrm{e}^{-\boldsymbol{U}^{(\boldsymbol{f})}_t\boldsymbol{h}(x)}$ defined by   
\[
\mathrm{e}^{-{U}^{(\boldsymbol{f},i)}_t\boldsymbol{h}(x)}:=\widehat{\mathbb{E}}_{\boldsymbol{\mu},\boldsymbol{e}_i \delta_x}\left[ \mathrm{e}^{-\langle \boldsymbol{f},\boldsymbol{I}^{\mathbb{N}^\dag}+\boldsymbol{I}^{\mathbb{P}^\dag}+\boldsymbol{I}^{\eta}_t\rangle-\langle \boldsymbol{h},\boldsymbol{Z}_t\rangle}\right],
\]
 is a solution to  \eqref{cthm1} and that this solution is unique.  By its definition, it is clear that $\mathrm{e}^{-\boldsymbol{U}^{(\boldsymbol{f})}_t\boldsymbol{h}(x)}\in[0,1]^\ell$ for all $x\in \R^d$ and $t\geq 0$. On the other hand from Lemma \ref{lemma1}, we observe that 
\[
\mathrm{e}^{-{U}^{(\boldsymbol{f},i)}_t\boldsymbol{h}(x)}=\widehat{\mathbb{E}}_{\boldsymbol{\mu},\boldsymbol{e}_i \delta_x}\left[ \exp\left\{-\langle \boldsymbol{f},\boldsymbol{I}^{\eta}_t\rangle-\langle \boldsymbol{h},\boldsymbol{Z}_t\rangle -\int_0^t \langle \boldsymbol{\phi}(\boldsymbol{V}^\dag_{t-r}\boldsymbol{f}), \boldsymbol{Z}_r\rangle \mathrm{d}r\right\}\right].
\]

Therefore  Lemma \ref{lemma2} implies that the vectorial function $\mathrm{e}^{-\boldsymbol{U}^{(\boldsymbol{f})}_t\boldsymbol{h}(x)}$  satisfies
\begin{equation*}
\begin{split}
\mathrm{e}^{-{U}^{(\boldsymbol{f},i)}_t\boldsymbol{h}(x)}=\mathtt{E}^{(i)}_x\left[\mathrm{e}^{-h_i(\xi^{(i)}_t)} +\frac{1}{w_i}\int_0^t  \Bigg(H_{t-s}^{(i)}(\xi^{(i)}_s,\boldsymbol{w}\cdot\mathrm{e}^{-\boldsymbol{U}^{(\boldsymbol{f},i)}_{t-s}\boldsymbol{h}(\xi^{(i)}_s)})-\phi(i,
\boldsymbol{V}_{t-r}^\dag\boldsymbol{f}(\xi_s^{(i)}))w_i\mathrm{e}^{-{U}^{(\boldsymbol{f},i)}_{t-s}\boldsymbol{h}(\xi^{(i)}_s)}\Bigg) \ud s \right],
\end{split}
\end{equation*}
where $H^{(i)}$ is given as in \eqref{H}. Using the definitions of $\boldsymbol{\psi}^{\dag},\ \boldsymbol{\phi} $ and  $H$  (see identities \eqref{phi} \eqref{psistar} and \eqref{H}), we observe  for all $i\in S$, $x\in\R^d$ and  $\boldsymbol{\theta}\in \R^l_+$, that
\begin{equation*}
\begin{split}
H^{(i)}_{t}(x,\boldsymbol{\theta})-\phi(i,\boldsymbol{V}^\dag_{t}\boldsymbol{f}(x))\theta_i =\psi^{\dag}\left(i, -\boldsymbol{\theta}+\boldsymbol{V}^\dag_{t}\boldsymbol{f}(x)\right)-\psi^{\dag}(i,\boldsymbol{V}^\dag_{t}\boldsymbol{f}(x)).
\end{split}
\end{equation*}
Therefore,  $\mathrm{e}^{-\boldsymbol{U}^{(\boldsymbol{f})}_t\boldsymbol{h}(x)}$  is a solution to  \eqref{cthm1}. 

In order to finish the proof, we show  that the solution to \eqref{cthm1} is unique.  Our arguments use  Gronwall's lemma and similar ideas to  those used in the monograph of Li \cite{Z} and in Proposition \ref{Proposition1}. With this purpose in mind, we first deduce some additional inequalities. Recall that the function $\psi^\dag(i,\boldsymbol{\theta})$ defined in \eqref{psistar} is a branching mechanism. Using similar notation as in  Proposition \ref{Proposition1}, we introduce the operator
$${\Psi}^\dag(x,i,f)=\psi^\dag(i,(f(x,1),\cdots,f(x,\ell))),$$
for  ${f}\in \mathcal{B}( \R^d\times S)$, and observe that it satisfies identity (2.26) in \cite{Z}. Therefore, following line by line the arguments in the proof of Proposition 2.20 in \cite{Z}, we may deduce that ${\Psi}^\dag$ satisfies Condition 2.11 in \cite{Z}. In other words, for all $a\geq 0$, there exists $L_a>0$ such that
\begin{equation}\label{condition 2.10}
\sup_{(x,i)\in \R^d\times S}|{\Psi}^\dag(x,i,f)-{\Psi}^\dag(x,i,g)|\leq L_a \|f-g\|,\qquad \textrm{for }\quad f,g\in \mathcal{B}_a(\R^d\times S),
\end{equation}
where $\|f\|:=\sup_{(x,i)\in \R^d\times S}|f(x,i)|$ and $\mathcal{B}_a(\R^d\times S):=\{f\in \mathcal{B}(\R^d\times S):\|f\|\leq a\}$.

On the other hand by Proposition 2.21 in \cite{Z},  for all $ f\in \mathcal{B}( \R^d\times S)$,  there exists $t\mapsto \mathtt{V}^\dag_tf$ a unique locally bounded positive solution to
\begin{equation*}
\mathtt{V}^\dag_tf(x, i)=\mathtt{T}_t f(x,i)-\int_0^t\ud s\int_{\R^d\times S}\Psi^\dag(y,j,\mathtt{V}^\dag_{t-s}{f})\mathtt{T}_s(x,i,\ud (y,j)),
\end{equation*}
where the semigroup $\mathtt{T}_t$ is given as in \eqref{ttdefinition}. Moreover, by Proposition 2.14 in \cite{Z}, for all $T>0$ there exists $C(T)$ such that 
\begin{equation*}
\sup_{0\leq s\leq T}\sup_{(x,i)\in \R^d\times S}|\mathtt{V}^\dag_sf(x, i)
|\leq C(T)\|f\|.
\end{equation*}
Hence  using the homeomorphism between $\mathcal{B}(\R^d)^\ell$ and  $\mathcal{B}(\R^d\times S)$ which was defined in the proof of Proposition \ref{Proposition1} (i.e. for  $\boldsymbol{f}\in\mathcal{B}(\R^d)^\ell$, we define $f(x,i)=f_i(x)$) and  the previous inequality, we deduce that 
\begin{equation}\label{bounded Vdag}
\sup_{0\leq s\leq T}\sup_{x\in \R^d}\sup_{i\in S}\Big|V^{\dag,(i)}_s\boldsymbol{f}(x)
\Big|\leq C(T)\|\boldsymbol{f}\|\qquad \qquad\textrm{for }\quad \boldsymbol{f}\in\mathcal{B}^+(\mathbb{R}^d)^\ell,
\end{equation}
where $\|\boldsymbol{f}\|=\sup_{x\in \R^d}\sup_{i\in S}|f_i(x)|$ and $\boldsymbol{V}^\dag\boldsymbol{f}$ is given by \eqref{semi-cond}.

Next, we take  $\mathrm{e}^{-\boldsymbol{W}_t(x)}$  and $\mathrm{e}^{-\widetilde{\boldsymbol{W}}_t(x)}$,    two solutions of \eqref{cthm1},  and observe that  for all $i\in S$

\begin{equation*}
\begin{split}
w_i\mathrm{e}^{-{W}^{(i)}_t(x)}-w_i\mathrm{e}^{-\widetilde{W}^{(i)}_t(x)}=\int_0^t\ud s\int_{\R^d} &  \left[\psi^\dag\left(i, -\boldsymbol{w}\cdot \mathrm{e}^{-\boldsymbol{W}_{t-s}(y)}
+\boldsymbol{V}^{\dag}_{t-s}\boldsymbol{f}(y)\right)\right.\\
&\hspace{1.5cm}\left.-\psi^\dag\left(i,-\boldsymbol{w}\cdot \mathrm{e}^{-\widetilde{\boldsymbol{W}}_{t-s}(y)}+\boldsymbol{V}^{\dag}_{t-s}\boldsymbol{f}(y)\right)\right] \mathtt{P}^{(i)}_s(x,\ud y).
\end{split}
\end{equation*}
Since $\mathrm{e}^{-\boldsymbol{W}_t(x)}\in [0,1]^\ell$ and $\boldsymbol{V}^\dag f$ satisfies \eqref{bounded Vdag}, we have, for all $s\leq T$, that
$$\Big\|-\boldsymbol{w}\cdot \mathrm{e}^{-\boldsymbol{W}_s(x)}+\boldsymbol{V}^\dag_s f(x)\Big\|\leq \|\boldsymbol{w}\|+C(T)\|f\|:=a(T),$$
 and the same inequality holds for $\mathrm{e}^{-\widetilde{\boldsymbol{W}}_t(x)}$. Therefore, by the definition of $\Psi^\dag$ and \eqref{condition 2.10}, there exists $L_T>0$ such that we obtain, for all $t\leq T$, the following inequality
\begin{equation*}
\begin{split}
\Big|w_i\mathrm{e}^{-{W}^{(i)}_t(x)}-w_i\mathrm{e}^{-\widetilde{W}^{(i)}_t(x)}\Big|\leq\int_0^t\ud s\int_{\R^d}  L_T  \Big\|\boldsymbol{w}\cdot \mathrm{e}^{-\boldsymbol{W}_{t-s}(x)}-\boldsymbol{w}\cdot \mathrm{e}^{-\widetilde{\boldsymbol{W}}_{t-s}(x)}\Big\| \mathtt{P}^{(i)}_s(x,\ud y).
\end{split}
\end{equation*}
The latter implies the following inequality
$$\Big\|\boldsymbol{w}\cdot \mathrm{e}^{-\boldsymbol{W}_{t}(x)}-\boldsymbol{w}\cdot \mathrm{e}^{-\widetilde{\boldsymbol{W}}_{t}(x)}\Big\| \leq L_T \int_0^t   \Big\|\boldsymbol{w}\cdot \mathrm{e}^{-\boldsymbol{W}_{s}(x)}-\boldsymbol{w}\cdot \mathrm{e}^{-\widetilde{\boldsymbol{W}}_{s}(x)}\Big\|\ud s, \qquad \mbox{ for all } t\leq T.$$
 Thus by Gronwall's inequality, we deduce that 
	$$\boldsymbol{w}\cdot \mathrm{e}^{-\boldsymbol{W}_{s}(x)}=\boldsymbol{w}\cdot \mathrm{e}^{-\widetilde{\boldsymbol{W}}_{s}(x)} \qquad \mbox{ for all }\quad s\leq T.$$
Finally, because $T>0$ was  arbitrary, we get the uniqueness of the solution to  \eqref{cthm1}. 
\end{proof}

\begin{proof}[Proof of Theorem \ref{mainthm}]
Recall that $((\boldsymbol{\Lambda},\boldsymbol{Z}),\widehat{\mathbb{P}}_{(\boldsymbol{\nu},\boldsymbol{\mu})})$  is a Markov process and that $\widehat{\mathbb{P}}_{\boldsymbol{\mu}}$ is defined as $\widehat{\mathbb{P}}_{(\widetilde{\boldsymbol{\nu}},\boldsymbol{\mu})}$,  where $\widetilde{\boldsymbol{\nu}}$  is such that  $\widetilde{\nu}_i$ is a Poisson random measure with intensity $w_i\mu_i$, for all  $i\in S$. Therefore, for $s, t\ge 0$, we see that 
	\begin{equation*}
	\widehat{\mathbb{E}}_{\boldsymbol{\mu}}\Big[f(\boldsymbol{\Lambda}_{t+s}) \Big|(\boldsymbol{\Lambda}_u, u\leq s)\Big]=\widehat{\mathbb{E}}_{\widetilde{\boldsymbol{\nu}},\boldsymbol{\mu}}\Big[f(\boldsymbol{\Lambda}_{t+s}) \Big|(\boldsymbol{\Lambda}_u, u\leq s)\Big]=\widehat{\mathbb{E}}_{(\boldsymbol{Z}_{s},\boldsymbol{\Lambda}_{s})}\Big[f(\boldsymbol{\Lambda}_{t})\Big].
	\end{equation*}
 	Then, in order to deduce that $(\boldsymbol{\Lambda},\widehat{\mathbb{P}}_{\boldsymbol{\mu}})$ is Markovian, we need to show that each coordinate of $\boldsymbol{Z}_t=(Z_t^1,\dots,Z_t^\ell)$ given $\boldsymbol{\Lambda}_t=(\Lambda_t^1,\dots,\Lambda_t^\ell)$ is a Poisson random measure with intensity $w_i\Lambda_t^i$. From Campbell's formula for Poisson random measures (see for instance Section 3.2 of \cite{kingman}), the latter  is equivalent to showing  that for all $\boldsymbol{h}\in \mathcal{B}^+(\mathbb{R}^d)^\ell$
\[
\widehat{\mathbb{E}}_{\boldsymbol{\mu}}\left[\left.\mathrm{e}^{-\langle \boldsymbol{h}, \boldsymbol{Z}_t\rangle}\right|\boldsymbol{\Lambda}_t\right]=\exp\left\{-\langle \boldsymbol{w}\cdot(\boldsymbol{1}-\mathrm{e}^{\boldsymbol{h}}),\boldsymbol{\Lambda}_t\rangle \right\},
\]
or equivalently,   that for all $\boldsymbol{f},\boldsymbol{h}\in  \mathcal{B}^+(\mathbb{R}^d)^\ell$
\begin{equation}\label{cond}
\widehat{\mathbb{E}}_{\boldsymbol{\mu}}\left[\mathrm{e}^{-\langle \boldsymbol{f},\boldsymbol{\Lambda}_t\rangle-\langle \boldsymbol{h},\boldsymbol{Z}_t\rangle}\right]=\widehat{\mathbb{E}}_{\boldsymbol{\mu}}\left[\mathrm{e}^{-\langle \boldsymbol{w}\cdot(\boldsymbol{1}-\mathrm{e}^{-\boldsymbol{h}})+\boldsymbol{f},\boldsymbol{\Lambda}_t\rangle}\right].
\end{equation}
Using   \eqref{laplceexpjoint} with $\widetilde{\boldsymbol{\nu}}$, we find
\[
\widehat{\mathbb{E}}_{\boldsymbol{\mu}}\left[\mathrm{e}^{-\langle \boldsymbol{f},\boldsymbol{\Lambda}_t\rangle-\langle \boldsymbol{h},\boldsymbol{Z}_t\rangle}\right]=\exp\left\{-\langle \boldsymbol{V}^\dag_t\boldsymbol{f}+\boldsymbol{w}\cdot(\boldsymbol{1}-\mathrm{e}^{-\boldsymbol{U}^{(\boldsymbol{f})}_t\boldsymbol{h}}),\boldsymbol{\mu}\rangle\right\}.
\]
Similarly,   considering \eqref{laplceexpjoint} again  with $\widetilde{\boldsymbol{\nu}}$,   $\widetilde{\boldsymbol{f}}=\boldsymbol{w}\cdot(\boldsymbol{1}-\mathrm{e}^{-\boldsymbol{h}})+\boldsymbol{f}$ and $\widetilde{\boldsymbol{h}}=\boldsymbol{0}$,  we get that 
\[
\widehat{\mathbb{E}}_{\boldsymbol{\mu}}\left[\mathrm{e}^{-\langle \boldsymbol{w}\cdot(\boldsymbol{1}-\mathrm{e}^{-\boldsymbol{h}})+\boldsymbol{f},\boldsymbol{\Lambda}_t\rangle}\right]=\exp\left\{-\left\langle \boldsymbol{V}^\dag_t(\boldsymbol{w}\cdot(\boldsymbol{1}-\mathrm{e}^{-\boldsymbol{h}})+\boldsymbol{f})+\boldsymbol{w}\cdot(\boldsymbol{1}-\mathrm{e}^{-\boldsymbol{U}^{(\boldsymbol{w}\cdot(\boldsymbol{1}-\mathrm{e}^{-\boldsymbol{h}})+\boldsymbol{f})}_t\boldsymbol{0}}),
\boldsymbol{\mu}\right\rangle\right\}.
\]
Hence, if we prove that for any $\boldsymbol{f},\boldsymbol{h}\in \mathcal{B}^+(\mathbb{R}^d)^\ell$, $x\in\mathbb{R}^d$, and $i\in S$, the following identity holds
\begin{equation}\label{cond2}
V_t^{\dag(i)}\boldsymbol{ f}(x)+w_i(1-\mathrm{e}^{-{U}^{(\boldsymbol{f},i)}_t\boldsymbol{h}(x)})=V_t^{\dag(i)} (\boldsymbol{w}\cdot(\boldsymbol{1}-\mathrm{e}^{-\boldsymbol{h}})+\boldsymbol{f})(x)+w_i\left(1-\mathrm{e}^{-{U}^{((\boldsymbol{w}\cdot(\boldsymbol{1}-\mathrm{e}^{-\boldsymbol{h}})+\boldsymbol{f}),i)}_t\boldsymbol{0}(x)}\right),
\end{equation}
we can deduce \eqref{cond}.  

In order to obtain \eqref{cond2}, we first observe that identities \eqref{semi-cond} and \eqref{cthm1} together with the definition of $ \boldsymbol{\psi}^\dag$ allow us to see that both left and right hand sides of \eqref{cond2} solve \eqref{inteq_u} with initial condition $\boldsymbol{f}+\boldsymbol{w}\cdot(\boldsymbol{1}-\mathrm{e}^{-\boldsymbol{h}})$. 
Since \eqref{inteq_u} has a unique solution, namely $\boldsymbol{V}_t(\boldsymbol{f}+\boldsymbol{w}\cdot(\boldsymbol{1}-\mathrm{e}^{-\boldsymbol{h}}))$, we conclude that \eqref{cond2} 
holds and it is equal to $V^{(i)}_t(\boldsymbol{f}+\boldsymbol{w}\cdot(\boldsymbol{1}-\mathrm{e}^{-\boldsymbol{h}}))(x)$. Hence,  we can  finally deduce that $(\boldsymbol{\Lambda}, \widehat{\mathbb{P}}_{\boldsymbol{\mu}})$ is a Markov process. Moreover, we have
\begin{equation*}
\widehat{\mathbb{E}}_{\boldsymbol{\mu}}\left[\mathrm{e}^{-\langle \boldsymbol{f},\boldsymbol{\Lambda}_t\rangle-\langle \boldsymbol{h},\boldsymbol{Z}_t\rangle}\right]=\mathrm{e}^{-\left\langle \boldsymbol{V}_t(\boldsymbol{f}+\boldsymbol{w}\cdot(\boldsymbol{1}-\mathrm{e}^{-\boldsymbol{h}})),\boldsymbol{\mu}\right\rangle}=\mathbb{E}_{\boldsymbol{\mu}}\left[\mathrm{e}^{-\langle \boldsymbol{f}+\boldsymbol{w}\cdot(\boldsymbol{1}-\mathrm{e}^{-\boldsymbol{h}}),\boldsymbol{X}_t\rangle}\right],
\end{equation*}
and if, in particular, we  take $\boldsymbol{h}=\boldsymbol{0}$ the above identity is reduced to
\begin{equation*}
\widehat{\mathbb{E}}_{\boldsymbol{\mu}}\left[\mathrm{e}^{-\langle \boldsymbol{f},\boldsymbol{\Lambda}_t\rangle}\right]=\mathbb{E}_{\boldsymbol{\mu}}\left[\mathrm{e}^{-\langle \boldsymbol{f},\boldsymbol{X}_t\rangle}\right].
\end{equation*}
This completes the proof.
\end{proof}

\bigskip
\noindent \textbf{Acknowledgements.} This project began while the authors were attending a Bath, UNAM, and CIMAT (BUC) workshop in Guanajuato, Mexico in May, 2016. The authors thank Andreas Kyprianou  and Victor Rivero for their roles in organising  this workshop. This research was supported by the Royal Society and CONACyT-MEXICO. DF is supported by a scholarship from the EPSRC Centre for Doctoral Training, SAMBa.  
\end{document}